\documentclass[12pt, a4paper, reqno]{amsart}
\usepackage[utf8]{inputenc}
\usepackage{lmodern}
\usepackage{upgreek}
\usepackage{xcolor}
\usepackage{amsfonts}
\usepackage{amsmath}
\usepackage{amssymb}
\usepackage{amsthm}
\usepackage{amsrefs}
\usepackage{comment}
\usepackage{etaremune}
\usepackage{enumitem}
\usepackage{pdfpages}
\usepackage{fancyhdr}
\usepackage{setspace}
\usepackage{hyperref}
\usepackage{mathrsfs}
\usepackage{footnote}
\hypersetup{
	colorlinks=true,
	linkcolor=blue,
	filecolor=blue,      
	urlcolor=blue,
	citecolor=blue,
	pdftitle={},
	pdfpagemode=FullScreen,}

\numberwithin{equation}{section}
\newtheorem{theorem}{Theorem}[section]
\newtheorem{lem}[theorem]{Lemma}

\newtheorem{thmm}{Theorem}

\newtheorem{prop}[theorem]{Proposition}

\theoremstyle{definition}
\newtheorem{re}{Remark}[section]
\newtheorem{definition}[theorem]{Definition}

\newcommand{\Mod}[1]{\ (\mathrm{mod}\ #1)}
\newcommand{\mc}[1]{\mathcal{#1}}

\newcommand{\mb}[1]{\mathbb{#1}}
\newcommand{\st}[1]{\substack{#1}}
\newcommand{\lr}[1]{\left(#1\right)}

\newcommand{\p}[1]{\Phi_a{(#1)}}

\newcommand{\vp}[1]{\varphi{(#1)}}

\title{Diophantine approximation with integers having no large prime factors}

\author{Kunjakanan Nath and Habibur Rahaman}

\address{Institut \'Elie Cartan de Lorraine, Universit\'e de Lorraine, CNRS, F-54000 Nancy, France}
\email{kunjakanan@gmail.com}

\address{Indian Institute of Science Education \& Research Kolkata, Department of Mathematics and Statistics, Mohanpur, West Bengal-741246, India}
\email{hr21rs044@iiserkol.ac.in}

\date{\today}

\begin{document}

\begin{abstract}
Given any irrational number $\alpha$, we show that for any $0<\theta<6/17$, there are infinitely many $y$-smooth (friable) numbers $n$ 
    such that 
    \[\|n\alpha\| < n^{-\theta},\]
where $(\log n)^C\leq y\leq n$ for some large constant $C>0$. This improves the previous work of Baker, who obtained the exponent $1/3-2/(3C)+o(1)$ in the case of $y\geq (\log n)^C$, and that of Yau, who obtained the exponent $1/3$ when $y=n^{o(1)}$. Our proof is based on the \emph{dispersion method} together with arithmetic inputs coming from the average bounds for Kloosterman sums over smooth numbers.
\end{abstract}


	\maketitle

    \section{Introduction}\label{sec: intro}

\subsection{Motivation and previous results} Given any irrational number $\alpha$, Dirichlet's approximation theorem implies that there are infinitely many positive integers $n$ such that 
\[\|n\alpha\|<1/n.\]
Here and throughout this paper, $\|x\|$ denotes the distance from $x$ to the nearest integer. A natural follow-up question is to investigate the situation when one restricts $n$ to lie in a particular subset of the positive integers. This question has sparked an active line of research exploring the beautiful interplay between the Diophantine approximation and the arithmetic structure of integers.

The most obvious and natural choice is to consider the set of all primes. Under the Generalised Riemann Hypothesis (GRH), it is known (see, for example, \cite[Theorem 2]{Bal1983}) that for any $\varepsilon>0$, there are infinitely many primes $p$ such that
\[\|p\alpha \|<p^{-1/3+\varepsilon}.\]
Matom\"aki \cite{Mat2009} proved this result unconditionally by using a \emph{dispersion method} and the averages of Kloosterman sums bound.  
In fact, there is a rich literature on this problem of Diophantine approximation with primes, starting with the work of Vinogradov \cite{Vin2004} and subsequent developments by Vaughan \cite{Vau1977}, Harman \cites{Har1983, Har1996}, and Heath-Brown and Jia \cite{HBJ2002}.

While the exponent $1/3$ appears to be a barrier in the case of primes, even under GRH, it is reasonable to ask if one can break it for \emph{almost primes}, that is, integers with a bounded number of prime factors. Using sieve methods and exponential sums bound, Harman \cite{Har1983} showed that there are infinitely many integers with at most two prime factors such that for any $\varepsilon>0$,
\[\|n\alpha\|\leq n^{-0.46+\varepsilon}.\]
In fact, more is known in this direction. Irving  \cite{Irv2014} showed that there are infinitely many positive integers $n$ which are the product of two distinct primes such that for any $\varepsilon>0$,
\[\|n\alpha\|<n^{-8/23+\varepsilon}.\]

In this paper, we are interested in the set of \emph{smooth numbers} or \emph{friable numbers}, that is, positive integers that have no large prime factors. From the multiplicative point of view, the set of smooth numbers lies on the opposite spectrum from primes. Given any $y\geq 2$, we say that a number $n$ is $y$-smooth if it has no prime factor greater than $y$. To the best of our knowledge, Yau \cite{Yau2019} appears to be the first to study the problem of Diophantine approximations by smooth numbers. He proved that for any fixed $\varepsilon>0$, there are infinitely many $n^\varepsilon$-smooth numbers $n$, such that 
\[\|n\alpha\|<n^{-1/3+\varepsilon}.\]
It is known that the smaller the value of $y$, the ``sparser'' the set of $y$-smooth numbers. Recently, Baker \cite{Bak2023} improved Yau's result to $(\log n)^C$-smooth numbers with $C>2$. More precisely, he proved that for any given $C>2$ and $\varepsilon>0$ arbitrarily small, there are infinitely many positive integers $n$ that are free of prime factors greater than $(\log n)^C$, such that
\[\|n\alpha \|<n^{-\left(\frac{1}{3}-\frac{2}{3C}\right)+\varepsilon}.\]

\subsection{Main result}
The goal of this paper is to establish the following result, which improves the exponent $1/3$ in the above results of Baker and Yau.

    \begin{thmm}\label{main_thm}
    Let $\alpha$ be a given irrational number.
    Then, for any $0<\theta<6/17$, there are infinitely many $Y$-smooth numbers $n$ 
    such that 
    \[\|n\alpha\|<n^{-\theta},\]
where $(\log n)^C\leq Y\leq n$ for some large absolute constant $C>0$.
\end{thmm}

\begin{re}
    The exponent $6/17=1/3+1/(51)$ is not at all optimal in the above result. It is natural to expect that the optimal exponent would be $1-\varepsilon$. We remark that the best exponent one could hope to obtain using our method is $2/5$, which in turn boils down to improving Theorem \ref{thm_kloos}.
\end{re}

We now briefly mention the key idea behind our improvement over Baker \cite{Bak2023} and Yau \cite{Yau2019}. Both Baker and Yau used the Fourier expansion of the periodic function $\|x\|$ to transform the Diophantine problem into one involving the exponential sums over smooth numbers. The exponential sums are then treated using Vinogradov's classical approach. This restricts them to take the exponent $\theta \approx 1/3$. Roughly speaking, their analysis reduces to sums of the form
\begin{align*}
    \sum_{1\leq h\leq H}\bigg|\sum_{\substack{mn\sim x\\ P^+(mn)\leq y\\ m\sim x^{1/3}}}e(mnh\alpha)\bigg|,
\end{align*}
where $H\approx x^{\theta}$. The goal is to save a factor of $H$ in the above sum. In particular, Vinogradov's exponential sum method gives non-trivial savings to the above sum as long as $H\leq x^{1/3-\varepsilon}$, which yields the exponent $\theta\approx 1/3$. Therefore, to beat the exponent $1/3$, the above approach of using exponential sums will not work.

 Our proof of Theorem \ref{main_thm} is inspired by the works of Heath-Brown and Jia \cite{HBJ2002} and Irving \cite{Irv2014}. To go beyond the exponent $1/3$, we will reduce our problem to counting smooth numbers in arithmetic progressions (see Section \ref{sec: Proof of main theorem} for a rigorous explanation). One important step in our proof is to reduce the problem to estimating the averages of Kloosterman sums over smooth numbers. This can be then estimated by applying works of Bettin-Chandee \cite[Theorem 1]{BC2018} or Duke-Friedlander-Iwaniec \cite[Theorem 2]{DFI1997}, which would yield exponents $20/59$ and $48/143$, respectively. However, in order to obtain the exponent $6/17$, we rely on a nice factorisation property of smooth numbers to establish new estimates on the average of Kloosterman sums over smooth numbers (see Theorem \ref{thm_kloos} for the full statement).

\subsection{Notation} We will use standard notation throughout this paper. However, for readers' convenience, we would like to highlight a few.  
\begin{itemize}
    \item Expressions of the form ${f}(x)=O({g}(x))$, ${f}(x) \ll {g}(x)$, and ${g}(x) \gg {f}(x)$ signify that $|{f}(x)| \leq C|{g}(x)|$ for all sufficiently large $x$, where $C>0$ is an absolute constant. A subscript of the form $\ll_A$ means the implied constant may depend on the parameter $A$. 
    \item The notation ${f}(x) \asymp {g}(x)$ indicates that ${f}(x) \ll {g}(x) \ll {f}(x)$.
    \item We write $f(x)=o(g(x))$ if $\lim_{x\to \infty}f(x)/g(x)=0$.
    \item  We denote $n\sim N$ to mean $N<n\leq 2N$.
    \item For any real number $x$, we will write $e(x):=e^{2\pi ix}$.
    \item We will denote the Euler totient function by $\varphi$.
     \item Given an integer $n\geq 1$, $\omega(n)$ denotes the number of distinct prime factors of $n$.
    \item We will set $(a, b)$ to be the greatest common divisor of integers $a$ and $b$ and by abuse of
notation, it will also denote the open interval on the real line. Its exact meaning will always be clear from the
context.
   \item Given a positive integer $n$, we let $P^+(n)$ to be the largest prime divisor of $n$, with the standard convention that $P^+(1)=1$.
    \item For any set $\mathcal{E}$, $\#\mathcal{E}$ denotes the cardinality of the set $\mathcal{E}$.
    \item For any set $\mathcal{E}$, we write
    \begin{align*}
        1_{\mathcal{E}}(x):=
        \begin{cases}
            1 & \text{if $x\in \mathcal{E}$},\\
            0 & \text{if $x\notin \mathcal{E}$}.
        \end{cases}
    \end{align*}
    \item Given any function $f\colon \mathbb{R}\to \mathbb{R}$, we define its Fourier transform $\widehat{f}\colon \mathbb{R}\to \mathbb{C}$ by
    \begin{align*}
        \widehat{f}(x):=\int_{-\infty}^\infty f(t)e(-xt)\: dt.
    \end{align*}
\end{itemize}

\subsection{Organization of the paper} The paper is organized as follows. In Section \ref{sec: Proof of main theorem}, we will reduce the proof of Theorem \ref{main_thm} to a counting problem and deduce its proof from Theorem \ref{main_thm2}. We will also provide a proof sketch of Theorem \ref{main_thm2} in Section \ref{sec: Proof of main theorem}.  We use Section \ref{sec: Prelim} to gather preliminary results on smooth numbers and exponential sum bounds. Next, we prove our result on the average of Kloosterman sums over smooth numbers, Theorem \ref{thm_kloos} in Section \ref{sec: kloos}.  In Section \ref{type1_sums}, we will prove the Type I sum estimate, namely, Proposition \ref{Prop: Main 1}. The bulk of this paper is devoted to establishing the Type II sum estimate, Proposition \ref{Prop: Main 2}, which we do in Section \ref{type2_sums}. Finally, in Section \ref{sec: Proof of main theorem 2}, we deduce Theorem \ref{main_thm2} from Proposition \ref{Prop: Main 1} and Proposition \ref{Prop: Main 2}.

\subsection*{Acknowledgements} The authors would like to thank Sary Drappeau, Sarvagya Jain, and Igor Shparlinski for their comments and suggestions on an earlier version of this paper, which strengthened our main result. The authors are also grateful to Sary Drappeau for pointing out the reference \cite{BT2005} and for helpful responses to emails regarding smooth numbers. 

 KN would like to thank Youness Lamzouri for some helpful discussions and encouragement.
 
HR would like to thank Stephan Baier for some helpful discussions and 
to the Institut \'Elie Cartan de Lorraine, Nancy, in particular, Youness Lamzouri, for his invitation and hospitality. HR is supported by the Prime Minister Research Fellowship (PMRF ID-0501972), funded by the Ministry of Education, Government of India.

\section{Set-up and proof of Theorem \ref{main_thm}}\label{sec: Proof of main theorem}

 We begin by noting that it is enough to establish our result for $Y=(\log n)^{C}$ for some large $C>0$. Indeed, if $n$ is $(\log n)^C$-smooth, then $n$ is $Y$-smooth for all $Y\geq (\log n)^C$. Therefore,
 \begin{align}\label{eq_elementary_obs}
     \{n\in \mathbb{N}\colon P^+(n)\leq (\log n)^C,\:  \|n\alpha\|\leq n^{-\theta}\}\subseteq \ \{n\in \mathbb{N}\colon P^+(n)\leq Y,\:  \|n\alpha\|\leq n^{-\theta}\}
 \end{align}
 for any $Y\geq (\log n)^C$. So, throughout this paper, we will concentrate on establishing the result for $Y=(\log n)^C$, where $C>0$ is some large constant.
 
 Let us now explain our strategy to prove the main result. The first step is to reduce the problem to counting smooth numbers in arithmetic progressions. To be precise, we begin with the following set-up. First, by Dirichlet's approximation theorem, we write
\begin{align}\label{def_alpha}
    \alpha=\dfrac{a}{q}+\beta,
\end{align}
where $(a, q)=1$ and $|\beta|\leq 1/q^2$. Next, throughout this paper, we set
\begin{align}\label{def_X}
    X=q^{2/(1+\theta)},
\end{align}
\begin{align}\label{def_R}
    R=q^{(1-\theta)/(1+\theta)},
\end{align}
where 
\begin{align}\label{def-theta}
    \theta=\dfrac{6}{17}-\varepsilon
\end{align}
for any fixed small $\varepsilon>0$. Note that $X= qR$.

With the above set-up, we now define the following set:
\begin{align*}
    S(X, Y, R, q):=\{X/4\leq n\leq 4X\colon P^+(n)\leq Y, \: (n, q)=1, \: an\equiv r\Mod q,  1\leq r\leq R\},
\end{align*}
where $Y=(\log X)^{C(\varepsilon)}$ for some large constant $C(\varepsilon)>0$.
We claim that if $n\in S(X, Y, R, q)$, then 
\[\|n\alpha\|\ll n^{-\theta}\]
holds, where $\theta$ is given by \eqref{def-theta}. Indeed, if $n\in S(X, Y, R, q)$, then $n$ is $Y$-smooth and $an=r+n_1q$, for some integers $r, n_1$ with $1\leq r\leq R$. So, by the triangle inequality, we have 
\begin{align}\label{Eq: connection}
    \|n\alpha\|=\Big{\|}\frac{an}{q}+n\beta\Big{\|}=\Big{\|}\frac{r}{q}+n\beta\Big{\|}\leq \dfrac{R}{q} + \dfrac{4X}{q^2}\ll \dfrac{1}{X^\theta}\ll \frac{1}{n^\theta}.
\end{align}
Therefore, our task reduces to showing that \[\#S(X,Y,R,q)\to \infty \quad \text{as $q\to \infty$}.\]

For technical convenience, we will consider a smooth version of the above counting problem. We therefore define the following Schwartz function (see, for example, \cite[Chapter~5]{SS2003}). 
\begin{definition}\label{def_W}
Let $\phi\colon \mathbb{R}\to [0,1]$ be a smooth function satisfying 
\begin{enumerate}
        \item $\phi(x)=0,$ for $x\notin[\frac{1}{4},\frac{3}{4}]$,
        \item $\phi(x)=1,$ for $x\in[\frac{1}{3},\frac{2}{3}]$, 
\item the $j^{th}$ derivative satisfies $\phi^{(j)}(x)\ll_j \min\lr{1, |x|^{-j}}$ for all $x\in \mathbb{R}$.
\end{enumerate}
\end{definition}

\begin{re}
We remark that applying integration by parts, we can deduce that for any integer $A>0$,
\begin{align*}
    |\widehat{\phi}(x)|=\bigg|\int_{-\infty}^\infty\phi(t)e(-xt)\: dt\bigg|\ll_A\min \{1, |x|^{-A}\},
\end{align*}
which we will use multiple times throughout this paper. Moreover, all implied constants may depend on $\phi$.
\end{re}

Equipped with the above definition of $\phi$, we consider the following sum
\begin{align}\label{def_sigma_q_R}
   \Sigma(q, R):= \sum_{\st{X/4\leq n\leq 4X}}1_{S_q(Y)}(n)\Phi_a(n, R),
\end{align}
where throughout this paper, we set
\begin{align}\label{def_ Phi_original}
    \Phi_a(n, R):=\sum_{\st{r\in\mb{Z}\\na\equiv r (\bmod{q})}}\phi\lr{\frac{r}{R}}
\end{align}
and
\begin{align}\label{def_S_q_Y}
    S_q(Y):=\{n\in \mathbb{N}\colon P^+(n)\leq Y, (n, q)=1\}.
\end{align}
We establish a lower bound for the sum $\Sigma(q, R)$ in the following theorem.

\begin{thmm}\label{main_thm2}
   Assume the above set-up. Let $\varepsilon>0$. Then, there exists a large real number $C(\varepsilon)>0$ such that as $q\to \infty$, we have
    \[\Sigma(q, R) \gg_{\varepsilon} R^{2-\frac{1-\theta}{2C(\varepsilon)}+o(1)}.\]
Here we use the standard convention that the above $o(1)$ term tends to $0$ as $q$ tends to infinity. In particular, the right-hand side of the above expression goes to infinity as $q\to\infty.$
\end{thmm}

We will prove Theorem \ref{main_thm2} in Section \ref{sec: Proof of main theorem 2}. Before that, we quickly demonstrate how to deduce Theorem \ref{main_thm} using it.

\begin{proof}[Proof of Theorem \ref{main_thm} assuming Theorem \ref{main_thm2}] 
First, note that it is enough to establish the result with the exponent $\theta$ given by \eqref{def-theta}.

Next, since $\alpha$ is irrational, by Dirichlet's approximation theorem, there are infinitely many positive integers $q$ such that $|\alpha-a/q|\leq 1/q^2$. Note that by \eqref{Eq: connection}, if $n\in S(X, Y, R, q)$, then  $\|n\alpha\|\ll n^{-\theta}$, where $\theta$ is given by \eqref{def-theta}. On the other hand, by the definition of $\phi$ and Theorem \ref{main_thm2}, there exists a large real number $C(\varepsilon)>0$ such that
$$\#S(X,Y,R,q)\gg \Sigma(q, R) \gg_{\varepsilon} R^{2-\frac{1-\theta}{2C(\varepsilon)} + o(1)}.$$
This implies that $\#S(X, Y, R, q)\to \infty$ as $q\to \infty$ as $R$ is a function of $q$ by \eqref{def_R}. Therefore, by our observation \eqref{eq_elementary_obs}, we can conclude that there are infinitely many $Y$-smooth numbers $n$ such that $\|n\alpha\|\ll n^{-\theta}$ holds where $(\log X)^{C(\varepsilon)}< Y\leq X$ and $\theta$ is given by \eqref{def-theta}.
\end{proof}

Hence, the task now reduces to establishing Theorem \ref{main_thm2}. We now give a proof outline of Theorem \ref{main_thm2}.

\subsection{Proof sketch of Theorem \ref{main_thm2}} We present here a non-rigorous sketch of the proof of Theorem \ref{main_thm2}. Recall that our goal is to estimate the sum
\begin{align*}
    \Sigma(q, R):=\sum_{X/4\leq n\leq 4X}1_{S_q(Y)}(n)\Phi_a(n, R),
\end{align*}
where $\Phi_a(n, R)$ is given by \eqref{def_ Phi_original} and $S_q(Y)$ is given by \eqref{def_S_q_Y} with $Y=(\log X)^{C(\varepsilon)}$ for some large $C(\varepsilon)>0$.

The basic idea in our approach is to use the nice factorisation property of smooth numbers and introduce a bilinear structure on the sum over $n$, so that (by abuse of notation) we can express it as $nm$, where $n\asymp N$ and $m\asymp M$ satisfy $NM\asymp X$, where $X$ is given by \eqref{def_X}. More precisely, let $M, N\geq 2$ be two real numbers such that  $X/4\leq MN\leq 4X$ and
\[\dfrac{q}{R^{1-\delta}}\leq N\leq R^{12/11-\delta},
\]
where $\delta>0$ is sufficiently small in terms of $\varepsilon$. Here the exponent $12/11$ is determined by the quantitative bounds available for the averages of Kloosterman sums over smooth numbers in Theorem \ref{thm_kloos} (see Lemma \ref{lem_for_S8} for precise requirements). The above condition implies that
\begin{align}\label{A2=>A1}
    M\leq 4X/N\leq 4XR^{1-\delta}/q\leq 4R^{2-\delta},
\end{align}
where we have used the fact that $X=qR$ in the last inequality. Then, by positivity, we have
\begin{align*}
    \Sigma(q, R)\gg \sum_{m\sim M}1_{S_q(Y)}(m)\sum_{n\sim N}1_{S_q(Y)}(n)\Phi_a(mn, R).
\end{align*}
Therefore, our task now reduces to estimating the following sum
\begin{align}\label{Def: B M N}
  \mc{B}(M, N):= \mathcal{B}(M, N, q, R):=\sum_{m\sim M}1_{S_q(Y)}(m)\sum_{n\sim N}1_{S_q(Y)}(n)\Phi(mn, R),
\end{align}
which we will deduce from our Type I sum (Proposition \ref{Prop: Main 1}) and Type II sum (Proposition \ref{Prop: Main 2}). We will apply Proposition \ref{Prop: Main 1} to show that
\begin{align*}
    \sum_{m\sim M}1_{S_q(Y)}(m)\sum_{n\sim N}\Phi_a(mn, R)\approx \dfrac{NR}{q}\sum_{m\sim M}1_{S_q(Y)}(m).
\end{align*}
The proof of this result follows in a straightforward manner by using Vinogradov's exponential sums method (see Lemma \ref{thm1_type1} for more details). To estimate the sum over $m$ restricted to smooth numbers with $(m, q)=1$, we employ results on smooth numbers given in Lemma \ref{lem_for_interval_small_y}.

On the other hand, Proposition \ref{Prop: Main 2} implies that for any $0<\eta<\delta/20$,
\begin{align*}
    \sum_{m\sim M}1_{S_q(Y)}(m)\sum_{n\sim N}\big(1_{S_q(Y)}(n)-K(N,Y)\big)\Phi_a(mn, R)\ll_{\delta} R^{2-\eta},
\end{align*}
where $K(x,y):=x^{-1}\sum_{n\sim x}1_{S_q(y)}(n)$. We now give a brief sketch of how we establish the above estimate. We will apply the \emph{dispersion method} to estimate the above `discrepancy sum'. Indeed, the application of the Cauchy-Schwarz inequality reduces to showing that
\begin{align*}
    \sum_{m}\phi\bigg(\dfrac{m}{3M}\bigg)\bigg|\sum_{n\sim N}\big(1_{S_q(Y)}(n)-K(N,Y)\big)\Phi_a(mn, R)\bigg|^2\ll\dfrac{R^{4-2\eta}}{M}.
\end{align*} 
Opening the square, it reduces to estimating sums of the form
\begin{align*}
    \sum_{n_1, n_2\sim N}f(n_1)g(n_2)\sum_{m}\phi\bigg(\dfrac{m}{3M}\bigg)\Phi_a(mn_1)\Phi_a(mn_2),
\end{align*}
where $f, g\in \{1, 1_{S_q(Y)}\}$ are two arithmetic functions. If $f=g=1$, we can deal with those sums using estimates from geometry of numbers (see Lemma \ref{thm3_type1}).

The main difficulty in our analysis lies in estimating sums of the form\footnote{Strictly speaking, this is not correct as we have to deal with the contributions coming from $(n_1, n_2)>1$. However, for this proof sketch, we will ignore such technicalities.}
\begin{align}\label{outline_sum}
  \sum_{\substack{n_1, n_2\sim N\\(n_1, n_2)=1}}1_{S_q(Y)}(n_1)g(n_2)\sum_{m}\phi\bigg(\dfrac{m}{3M}\bigg)\Phi_a(mn_1)\Phi_a(mn_2),
\end{align}
 where $g\in \{1, 1_{S_q(Y)}\}$ is an arithmetic function. Recall that $\Phi_a(n, R)$ is given by \eqref{def_ Phi_original}. Trivially $\Phi_a(n, R)\leq 1$. However, it is reasonable to expect that $\Phi_a(n, R)\approx R/q$ at least on average over $n$. Since $R<q$, we may apply the Poisson summation formula so that
\begin{align*}
    \Phi_a(n, R)=\dfrac{R}{q}\sum_{k\in \mathbb{Z}}\widehat{\phi}\bigg(\dfrac{kR}{q}\bigg)e\bigg(\dfrac{nak}{q}\bigg).
\end{align*}
Therefore, to estimate the above sum in \eqref{outline_sum}, we apply Poisson summation three times so that the inner sum over $m$ is roughly
\begin{align*}
\approx \frac{3MR^2}{q^2}\sum_{k_1}\sum_{k_2}\widehat{\phi}\lr{\frac{k_1R}{q}}\widehat{\phi}\lr{\frac{k_2R}{q}}\sum_{m}\widehat{\phi}\lr{3Mm-\frac{3Ma(k_1n_1+k_2n_2)}{q}}.
\end{align*}
The terms with $k_1n_1+k_2n_2=0$ will give an expected main term of the form
\begin{align*}
    \approx \dfrac{3\widehat{\phi}(0)^3MR}{q^2}\bigg(\sum_{n_1\sim N}1_{S_q(Y)}(n_1)\bigg)\bigg(\sum_{n_2\sim N}g(n_2)\bigg)
\end{align*}
and the error term can be easily controlled using the nice decay properties of the Fourier transform of $\phi$. On the other hand, for terms with $k_1n_1+k_2n_2\neq 0$, we make a change of variables $k_1n_1+k_2n_2=vq-a'u=c$, where $1\leq a'\leq q$ is an integer satisfying $aa'\equiv 1\Mod q$ and
\[m-\dfrac{a(k_1n_1+k_2n_2)}{q}=\dfrac{u}{q}.\]
This reduces our task to estimating the sum of the form
\begin{align*}
    \sum_{\substack{u, v, k_1, k_2\in \mathbb{Z}\\ k_1n_1 + k_2n_2=c\neq 0}}\widehat{\phi}\lr{\frac{k_1R}{q}}\widehat{\phi}\lr{\frac{k_2R}{q}}\widehat{\phi}\bigg(\dfrac{3Mu}{q}\bigg).
\end{align*}
 To detect the condition $k_1n_1+k_2n_2=c$, we apply Poisson summation to obtain
 \begin{align*}
     \sum_{\substack{k_1, k_2\in \mathbb{Z}\\ k_1n_1 + k_2n_2=c\neq 0}}\widehat{\phi}\lr{\frac{k_1R}{q}}\widehat{\phi}\lr{\frac{k_2R}{q}}\approx\dfrac{1}{n_2}\sum_{w\in \mathbb{Z}}\widehat{g}(w/n_2; n_1, n_2, c)e\bigg(\dfrac{c\overline{n_1}w}{n_2}\bigg),
 \end{align*}
 where $n_1\overline{n_1}\equiv 1 (\bmod{n_2})$ and $g(t; n_1, n_2, c)=\widehat{\phi}(tR/q)\widehat{\phi}((c-tn_1)R/q)$. Again, using the decay property of the Fourier transform of $\phi$, we can show that the contribution from terms with $w=0$ is small. However, the key challenge is to estimate the contributions coming from $w\neq 0$. In this case, using the decay property of the Fourier transform of $\phi$, we may restrict the variables $u, v, w$ in certain boxes, and our analysis reduces to estimating the following sum
 \begin{align*}
     \sum_{n_2\sim N}\bigg|\sum_{\substack{n_1\sim N\\(n_1, n_2)=1}}1_{S_q(Y)}(n_1)e\bigg(\dfrac{c\overline{n_1}w}{n_2}\bigg)\bigg|,
 \end{align*}
 where $|cw|$ is roughly bounded above by $N^2$. To obtain non-trivial estimates for the above sum, we apply Theorem \ref{thm_kloos}. In particular, any sharp non-trivial estimates for the above quantity would give an exponent of the form $1/3+\delta$ for some $\delta>0$ in our main result, Theorem \ref{main_thm}.

\section{Preliminaries}\label{sec: Prelim}

\subsection{Auxiliary results on smooth numbers}
In this section, we collect some technical results on smooth numbers that will be needed to prove our result. We begin with some standard notation. For $2\leq y\leq x$, let $S(x,y)$ denote the set of all $y$-smooth numbers not exceeding $x$, that is,
\begin{align}\label{def_smooth_set}
    S(x,y):=\{1\leq n\leq x\colon  P^+(n)\leq y\}. 
\end{align}
Next, for any integer $q\geq 1$, we set
\begin{align}\label{def_Sq(x,y)}
    S_q(x,y):=\{n\in S(x, y)\colon (n,q)=1\}. 
\end{align}
 Throughout this paper, we denote the cardinality of $S(x,y)$ and $S_q(x,y)$ by $\Psi(x,y)$ and $\Psi_q(x,y)$, respectively.

 We begin with the following estimates on smooth numbers. Let $u=\log x/\log y$. Then, it is known (see, for example, Hildebrand-Tenenbaum \cite[Corollary 1.3]{HT1993}) that for any $\varepsilon>0$, we have
 \begin{align}\label{Eq: smooth crude}
     \Psi(x, y)=xu^{-(1+o(1))u},
 \end{align}
 as $y, u\to \infty$, uniformly in the range $u\leq y^{1-\varepsilon}$. In fact, more precise estimates are known for $\Psi(x,y)$ in certain ranges of $y$. Indeed, for any fixed $\varepsilon>0$ and 
 \begin{align}\label{Eq: range_y_Hildebrand}
     \exp((\log \log x)^{5/3+\varepsilon})\leq y\leq x,
 \end{align} 
 Hildebrand \cite[Theorem 1]{Hil1986} showed that
 \begin{align}\label{smooth_Hil_global}
     \Psi(x,y)=x\rho(u)\bigg\{1+O_{\varepsilon}\lr{\frac{\log(u+1)}{\log y}}\bigg\}.
 \end{align}
  Here and throughout this paper, $\rho(u)$ is the Dickman–de Bruijn function, that is, the unique continuous solution of the following delay differential equation
    \begin{align*}
       \rho(u)=1\: (0\leq u\leq 1); \quad u\rho^\prime(u)=-\rho(u-1)\: (u>1).
    \end{align*}
       It is well-known (see, for example, \cite[Theorem III.5.13]{Ten1995}) that 
    \begin{align}\label{size_of_rho}
    \rho(u)=e^{-u\log u -u\log \log (2u) + O(u)}.
    \end{align}
Furthermore, by \cite[Theorem III.5.22]{Ten1995}, uniformly for $x\geq y\geq 2$, one has
\begin{align}\label{eq_ten_hil_Psi_2x}
    \Psi(2x, y)=\Psi(x,y)2^{\alpha(x,y)}\bigg\{1 + O\bigg(\dfrac{1}{u} + \dfrac{\log y}{y}\bigg)\bigg\},
\end{align}
where $\alpha(x,y)$ is the unique solution of the transcendental equation
    \begin{align}\label{def_alpha_x_y}
        \sum_{p\leq y}\dfrac{\log p}{p^{\alpha(x,y)}-1}=\log x.
    \end{align}
    By \cite[III.Eq (5.74)]{Ten1995}, it is known that for any $\varepsilon>0$ and $(\log x)^2\leq y\leq x$,
    \begin{align}\label{estimates_for_alpha_x_y}
        \alpha(x, y)=1-\dfrac{\log (u\log u)}{\log y} + O_\varepsilon\bigg(\dfrac{\log \log u}{(\log u)(\log y)} + \dfrac{1}{(\log x)(\log y)} + \exp\big(-(\log y)^{3/5-\varepsilon}\big)\bigg).
    \end{align}
    We remark that if $y=(\log x)^C$ for $C\geq 10$ (say), then the above expression implies that
    \begin{align}\label{alpha_x_log x}
    \alpha(x, (\log x)^C)= 1-\dfrac{1}{C} +o(1),
        \end{align}
        which we will use frequently.

 We will also need estimates for $\Psi_q(x, y)$ in terms of $\Psi(x, y)$, which we record from the work of La Bret\`eche-Tenenbaum \cite{BT2005} below.

\begin{lem}\label{thm2_smooth_asymp}
    Let $\varepsilon>0$ be fixed and let $x\geq 2$. Suppose that 
   \begin{align} 
   (\log x)^4\leq y\leq x, \quad  \omega(q)\ll \log x, \quad \text{and} \quad P^+(q)\leq y.
   \end{align}
   Then, we have
    \begin{align*}\label{cond_y_BT}
        \Psi_q(x, y)=\Psi(x,y)\prod_{p|q}\bigg(1-\dfrac{1}{p^{\alpha(x,y)}}\bigg)\bigg\{1 + O\bigg(\dfrac{1}{\log u}\bigg)\bigg\},
    \end{align*}
    where $\alpha(x,y)$ is given by \eqref{def_alpha_x_y}.
\end{lem}
\begin{proof}
We apply \cite[Th\'eor\`eme 2.1, Corollaire 2.2]{BT2005} to note that
    \begin{align*}
        \Psi_q(x, y)=\Psi(x,y)\prod_{p|q}\bigg(1-\dfrac{1}{p^\alpha}\bigg)\bigg\{1 + O\bigg(\dfrac{\big(\log \omega(q)\big)^2}{(\log u)(\log y)^2}\bigg)\bigg\}.
    \end{align*}
    Since $\omega(q)\ll \log x\leq \sqrt{y}$, the desired claim follows.
\end{proof}

We now estimate $Y$-smooth numbers coprime to $q$ in intervals. 

\begin{lem}\label{lem_for_interval_small_y}
Let $\varepsilon>0$ be fixed and $C(\varepsilon)>0$ be a large real number. Suppose that $X, q\geq 2$ be as in \eqref{def_X} and $Y=(\log X)^{C(\varepsilon)}$.
Let $X^\gamma\ll N\ll X$ for some $0<\gamma<1$. Then, we have 
\begin{align*}
\sum_{n\sim N}1_{S_q(Y)}(n)\gg_{\varepsilon, \: \gamma} \Psi(N, Y)\prod_{\substack{p|q\\ p\leq Y}}\bigg(1-\dfrac{1}{p^{\alpha(N, Y)}}\bigg),
\end{align*}
where $\alpha(N, Y)$ is given by \eqref{def_alpha_x_y}.
\end{lem}

\begin{proof}
Let $q_Y$ denote the $Y$-smooth part of $q$, that is,
    \begin{align*}
        q_Y:=\prod_{\substack{p^v \parallel q\\ p\leq Y }}p^v.
    \end{align*}
    Note that if $P^+(n)\leq Y$, then $(n, q)=1$ if and only if $(n, q_Y)=1$. 
   Since $q\ll X$, we have $\omega(q)\ll \log X$. So, we may apply Lemma \ref{thm2_smooth_asymp} to deduce that
    \begin{align*}
        \sum_{n\sim N}1_{S_q(Y)}(n)= &\: \bigg\{\Psi(2N, Y)\prod_{p|q_Y}\bigg(1-\dfrac{1}{p^{\alpha(2N, Y)}}\bigg) - \Psi(N, Y)\prod_{p|q_Y}\bigg(1-\dfrac{1}{p^{\alpha(N, y)}}\bigg)\bigg\}\\
        & \times \bigg\{1 + O\bigg(\dfrac{1}{\log (\log N/\log Y)}\bigg)\bigg\}.
    \end{align*}
Note that $\log N/\log Y \geq \gamma \log X/(C(\varepsilon)\log \log X)$. Finally, we may apply \eqref{eq_ten_hil_Psi_2x} together with the facts that $\alpha(N, y)\geq 1/2$ and $\alpha(2N, y)\sim \alpha(N, y)$ from \eqref{estimates_for_alpha_x_y} and \eqref{alpha_x_log x} to conclude that
\begin{align*}
    \sum_{n\sim N}1_{S_q(Y)}(n)\gg \Psi(N, Y)\prod_{\substack{p|q\\ p\leq Y}}\bigg(1-\dfrac{1}{p^{\alpha(N, Y)}}\bigg),
\end{align*}
as desired.
\end{proof}

We conclude our discussion on smooth numbers with the following factorisation property.

\begin{lem}\label{smooth_decomp}
    Suppose that $2\leq y\leq z<n\leq x$ and $n\in S(x,y)$. Then there exists a unique triple $(p,u,v)$ of integers with $p$ prime such that
    \begin{align*}
        n=uv, \quad u\in S(x/v, p), \quad z<v\leq zp, \quad p|v, \quad \text{and} \quad \text{if a prime $r|v$, then $p\leq r\leq y$}.
    \end{align*}
\end{lem}
\begin{proof}
    See \cite[Lemma~10.1]{Vau1989}.
\end{proof}

\subsection{Exponential sum bounds}
We will need the following bounds for exponential sums over any interval to estimate the averages of Kloosterman sums.
\begin{lem}\label{lem_exp}
    Let $\eta>0$ be arbitrarily small. For any integers $b,c$ with $c>1$ and real numbers $Z_1<Z_2$, we have 
    $$\sum_{\st{Z_1<n\leq Z_2\\(n,\: c)=1 }}e\lr{\frac{b\overline{n}}{c}}\ll_\eta c^\eta\lr{(b,c)\lr{\frac{Z_2-Z_1}{c}+1}+c^{1/2}},$$
    where $\overline{n}$ denotes the inverse of $n$ modulo $c$, that is, $n\overline{n}\equiv 1\pmod c$.
\end{lem}
\begin{proof}
   The proof follows from the Weil bound for the complete Kloosterman sums together with the standard completion method (see, for example, \cite[Lemma~1]{FS2011}).
\end{proof}

\subsection{Elementary bounds}
We also note the following elementary lemma, which will be useful later on.
\begin{lem}\label{lem_gcd_bound}
    Let $U\geq 2$ be large, and $k\ll U$ be any integer. Then, for any $\eta>0$, we have 
    \begin{align*}
       \sum_{\st{u_1,\: u_2\sim U\\u_1\neq u_2\\(u_1u_2,\: q)=1}}(u_1-u_2, ku_1u_2)\ll  \frac{\vp{q}}{q} U^{2+\eta}.
    \end{align*}
\end{lem}
\begin{proof}
    We write $u_1-u_2=u$, so that
    \begin{align*}
        \sum_{\st{u_1, \: u_2\sim U\\u_1\neq u_2\\(u_1u_2,q)=1}}(u_1-u_2, ku_1u_2)\ll& \sum_{\st{u_2\sim U\\(u_2,q)=1}}\sum_{0<u\ll U}(u, ku_2(u+u_2))= \sum_{\st{u_2\sim U\\(u_2,q)=1}}\sum_{0<u\ll U}(u, ku^2_2).
    \end{align*}
    Then the lemma follows from observing that, for any $n$, by the M\"obius inversion,
    $$\sum_{1\leq u\leq U}(u,n)=\sum_{d|n}\sum_{u\leq U/d}d\leq \sum_{d|n}U\ll Un^{\eta},$$
    by using the fact that $\sum_{d|n}1\ll n^\eta$ for any $\eta>0$.
\end{proof}

\section{Estimates for averages of Kloosterman sums}\label{sec: kloos}

Given two sequences of complex numbers, $\{\alpha_m\}_{m\geq 1}, \{\beta_n\}_{n\geq 1}$, Duke, Friedlander, and Iwaniec \cite[Theorem 2]{DFI1997} showed that for any non-zero integer $a$ and for any $\varepsilon>0$,
    \[\bigg|\sum_{m\sim M}\sum_{\st{n\sim N\\(m, \: n)=1}}\alpha_m\beta_ne\lr{\frac{a\overline{n}}{m}}\bigg|\ll_{\varepsilon} \big(\sum_{m\sim M}|\alpha_m|^2\big)^{1/2}\big(\sum_{n\sim N}|\beta_n|^2\big)^{1/2} (|a|+MN)^{3/8}(M+N)^{11/48+\varepsilon},\]
where throughout this section $\overline{n}$ denotes the inverse of $n$ modulo $m$, that is, $n\overline{n}\equiv 1\Mod m$.
When $M\asymp N, a\ll MN$, the above bound gives a saving of $N^{1/48}$ over the ``trivial bound''. The above result was extended to trilinear forms involving Kloosterman fractions, with an additional average over $a$, by Bettin and Chandee \cite{BC2018}. In particular, when $M\asymp N, a\ll MN$, the above result of Bettin-Chandee gives a saving of $N^{1/20}$ over the trivial bound.

Our goal is to obtain non-trivial estimates for the following sum
\begin{align*}
   \mathrm{Kl}_y(x; a, q):= \sum_{m\sim x}\bigg|\sum_{\substack{n<x\\ P^+(n)\leq y\\ (n, \: mq)=1}}e\bigg(\dfrac{a\overline{n}}{m}\bigg)\bigg|.
\end{align*}
In fact, we will establish a more general sum of the form
\begin{align}\label{def_B}
  \mathrm{Kl}_y(M, x; a, q):= \sum_{m\sim M}\bigg|\sum_{\st{n< x\\ P^+(n)\leq y\\(n, \: mq)=1}}e\lr{\frac{a\overline{n}}{m}}\bigg|
\end{align}
below, which might be of independent interest. 

\begin{thmm}\label{thm_kloos}
Let $M\geq 2$ be a real number, let $q\geq 1$ be an integer, and let $a$ be a non-zero integer. Suppose that $z$ is a real number satisfying $2\leq y\leq z<x$. Then, for any small $\eta>0$, we have
    \begin{align*}
   \mathrm{Kl}_y &(M, x; a, q)\ll_\eta  (|a|xM)^{\eta} \bigg(1+\dfrac{|a|}{xM}\bigg)^{1/2}(Mx^{1/2}y^{1/2}z^{1/2} + x^{3/2}M^{1/2}z^{-1/4}) + Mz.
    \end{align*}
\end{thmm}

\begin{re}\label{Remark: Consequence of Kloosterman sum}
    We will apply the above theorem when $M=x$ and $|a|\ll x^2$ together with an appropriate choice of $z$ and the range of $y=(\log x)^C$ for some large real number $C$. This will imply that the bound on the right-hand side is $\ll x^{2-\tau+\varepsilon}$ for any $\tau<1/6$. Qin and Zhang studied the above problem with an additional condition $\max (a, m)=1$ (see, for example, \cite[Theorem~2.7]{QZ2018}), which would then yield $\tau<1/10$.
\end{re}

\begin{proof}[Proof of Theorem \ref{thm_kloos}]
Let $z$ be a real number such that $2\leq y\leq z<x$. Then, by the triangle inequality, we have
\begin{align}\label{Kloos initial step}
\mathrm{Kl}_y(M, x; a, q)\leq \sum_{m\sim M}\bigg|\sum_{\substack{z<n<x\\ P^+(n)\leq y\\ (n, \: mq)=1}}e\bigg(\dfrac{a\overline{n}}{m}\bigg)\bigg| + O(Mz).
    \end{align}
    Next, by Lemma \ref{smooth_decomp}, we write
    \begin{align*}
        \sum_{\substack{z<n<x\\ P^+(n)\leq y\\ (n, \: mq)=1}}e\bigg(\dfrac{a\overline{n}}{m}\bigg)=\sum_{\substack{p\leq y\\ p\: \text{prime}}}C_y(p, x, z; a, q),
    \end{align*}
    where
    \begin{align*}
        C_y(p, x, z; a, q)=\sum_{\substack{z<v\leq zp\\ p|v\\ \text{prime}\: r|v\implies p\leq r\leq y\\ (v, \: mq)=1}}\sum_{\substack{z/v\leq u<x/v\\ P^+(u)\leq p\\ (u, \: mq)=1}}e\bigg(\dfrac{a\overline{uv}}{m}\bigg).
    \end{align*}
Now we write $v=pw$, so that 
\begin{align*}
    |C_y(p, x, z; a, q)|\leq \sum_{\substack{z/p<w\leq z\\ (wp, mq)=1}}\bigg|\sum_{\substack{z/(pw)\leq u<x/(pw)\\ P^+(u)\leq p\\ (u, \: mq)=1}}e\bigg(\dfrac{a\overline{upw}}{m}\bigg)\bigg|.
\end{align*}
Therefore, we infer that
\begin{align}\label{Introduction of K(p,w)}
    \sum_{m\sim M}\bigg|\sum_{\substack{z<n<x\\ P^+(n)\leq y\\ (n, \: mq)=1}}e\bigg(\dfrac{a\overline{n}}{m}\bigg)\bigg|
    \leq \sum_{p\leq y}\sum_{z/p<w\leq z}K(p, w),
\end{align}
where
\begin{align}\label{def K(p, w)}
    K(p, w):=\sum_{\substack{m\sim M\\ (m,\: pw)=1}}\bigg|\sum_{\substack{z/(pw)\leq u<x/(pw)\\ P^+(u)\leq p\\ (u, \: mq)=1}}e\bigg(\dfrac{a\overline{upw}}{m}\bigg)\bigg|.
\end{align}
To proceed, we consider two cases: whether $(a, upw)=1$ or $(a, upw)>1$. 

First, we consider the case $(a, upw)=1$. By the Cauchy-Schwarz inequality,
\begin{align*}
    K(p, w)^2
    &\leq M\sum_{\substack{m\sim M\\ (m,\: pw)=1}}\bigg|\sum_{\substack{z/(pw)\leq u<x/(pw)\\ P^+(u)\leq p\\ (u, \: mq)=1}}e\bigg(\dfrac{a\overline{upw}}{m}\bigg)\bigg|^2\\
    &\leq M\sum_{\substack{z/(pw)\leq u_1,\: u_2< x/(pw)\\ P^+(u_1), \: P^+(u_2)\leq p\\ (u_1u_2, q)=1}}\bigg|\sum_{\substack{m\sim M\\ (ma, \: pwu_1u_2)=1}}e\bigg(\dfrac{a(u_1-u_2)\overline{pwu_1u_2}}{m}\bigg)\bigg|\\
    &\ll M(\log x)^2 \max_{z/(pw)\leq U\leq x/(pw)}\sum_{\substack{u_1, \: u_2\sim U\\ (u_1u_2, aq)=1}}\bigg|\sum_{\substack{m\sim M\\ (m, \: pwu_1u_2)=1}}e\bigg(\dfrac{a(u_1-u_2)\overline{pwu_1u_2}}{m}\bigg)\bigg|,
\end{align*}
by a dyadic decomposition. Since $(m, pwu_1u_2)=1$, by B\'ezout's identity, we have
\begin{align*}
    \dfrac{\overline{pwu_1u_2}}{m} + \dfrac{\overline{m}}{pwu_1u_2}\equiv \dfrac{1}{mpwu_1u_2}\Mod 1,
\end{align*}
where $\overline{pwu_1u_2}$ and $ \overline{m}$ are the inverses modulo $m$ and $pwu_1u_2$, respectively. This implies that
\begin{align*}
    K(p, w)^2
    \ll M(\log x)^2\max_{z/(pw)\leq U\leq x/(pw)}\sum_{\substack{u_1, \: u_2\sim U\\ (u_1u_2, aq)=1}}\bigg|\sum_{\substack{m\sim M\\ (m, \: pwu_1u_2)=1}}e\bigg(\dfrac{-a(u_1-u_2)\overline{m}}{pwu_1u_2}\bigg)e\bigg(\dfrac{a(u_1-u_2)}{mpwu_1u_2}\bigg)\bigg|.
\end{align*}
Next, we remove the factor $e(a(u_1-u_2)/(mpwu_1u_2))$ by partial summation to obtain
\begin{equation}\label{K(p,w)}
\begin{aligned}
    K(p, w)^2\ll&\:  M(\log x)^2 \max_{z/(pw)\leq U\leq x/(pw)}\bigg(1 + \dfrac{|a|}{pwUM}\bigg)\\
    &\times \sum_{\substack{u_1, \: u_2\sim U\\ (u_1u_2, aq)=1}}\max_{M'\sim M}\bigg|\sum_{\substack{M\leq m<M'\\ (m, \: pwu_1u_2)=1}}e\bigg(\dfrac{-a(u_1-u_2)\overline{m}}{pwu_1u_2}\bigg)\bigg|
\end{aligned}
\end{equation}
The contribution from the diagonal terms $u_1=u_2$ in the above sum is
\begin{align}\label{K(p, w) diag}
    \ll M^2(\log x)^2\max_{z/(pw)\leq U\leq x/(pw)}U\bigg(1 + \dfrac{|a|}{pwUM}\bigg).
\end{align}
For the off-diagonal terms $u_1\neq u_2$ with $u_1, u_2\sim U$, we apply Lemma \ref{lem_exp} to obtain
\begin{align*}
    \max_{M'\sim M}\bigg| & \sum_{\substack{M\leq m<M'\\ (m, \: pwu_1u_2)=1}}e\bigg(\dfrac{-a(u_1-u_2)\overline{m}}{pwu_1u_2}\bigg)\bigg|\\
    & \ll_\eta (pwU^2)^\eta \bigg\{\big(a(u_1-u_2), pwu_1u_2\big)\bigg(1 + \dfrac{M}{pwU^2}\bigg) + (pw)^{1/2}U\bigg\}.
\end{align*}
for any $\eta>0$. Therefore,
\begin{equation}\label{K(p,w) off diagonal 1}
\begin{aligned}
    \sum_{\substack{u_1, \: u_2\sim U\\ u_1\neq u_2\\ (u_1u_2, aq)=1}}
    & \max_{M'\sim M}\bigg|\sum_{\substack{M\leq m<M'\\ (m, \: pwu_1u_2)=1}}e\bigg(\dfrac{-a(u_1-u_2)\overline{m}}{pwu_1u_2}\bigg)\bigg|\\
   &\ll_\eta x^{2\eta}\bigg\{\bigg(1+\dfrac{M}{pwU^2}\bigg)\sum_{\substack{u_1, u_2\sim U\\ u_1\neq u_2\\(u_1u_2, q)=1}}\big((u_1-u_2), pwu_1u_2\big)+ (pw)^{1/2}U^3 \bigg\},
\end{aligned}
\end{equation}
where we have used the fact that $(a, pwu_1u_2)=1$. Invoking Lemma \ref{lem_gcd_bound}, we see that the above expression is
\begin{align}\label{K(p, w) off diagonal 2}
    \ll_\eta x^{3\eta}\bigg\{\bigg(1 + \dfrac{M}{pwU^2}\bigg)U^2 + (pw)^{1/2}U^3\bigg\}.
\end{align}
Combining \eqref{K(p,w)}, \eqref{K(p, w) diag}, \eqref{K(p,w) off diagonal 1}, and \eqref{K(p, w) off diagonal 2}, we have
\begin{align*}
    K(p, w)^2 &\ll_\eta x^{4\eta}M\max_{z/(pw)\leq U\leq x/(pw)}\bigg(1 + \dfrac{|a|}{pwUM}\bigg)\bigg\{MU + U^2 + \dfrac{M}{pw} + (pw)^{1/2}U^3\bigg\}\\
    & \ll_\eta x^{4\eta}M\max_{z/(pw)\leq U\leq x/(pw)}\bigg(1 + \dfrac{|a|}{pwUM}\bigg)\big(MU + (pw)^{1/2}U^3\big).
\end{align*}
Therefore, from \eqref{Introduction of K(p,w)}, we can deduce that
\begin{align*}
    \sum_{m\sim M} & \bigg|\sum_{\substack{z<n<x\\ P^+(n)\leq y\\ (n, \: mq)=1}}e\bigg(\dfrac{a\overline{n}}{m}\bigg)\bigg|\\
    & \ll_\eta  x^{2\eta}M^{1/2}\sum_{p\leq y}\sum_{z/p<w\leq z}\max_{z/(pw)\leq U\leq x/(pw)}\bigg(1+\dfrac{|a|}{pwUM}\bigg)^{1/2}U^{1/2}\big(M^{1/2}+ (pw)^{1/4}U\big)\\
    &\ll_\eta x^{2\eta}M^{1/2}\sum_{p\leq y}\sum_{z/p<w\leq z}\bigg(\dfrac{x}{pw} + \dfrac{|a|}{pwM}\bigg)^{1/2}\bigg(M^{1/2} + \dfrac{x}{(pw)^{3/4}}\big)\\
    &\ll_{\eta}x^{2\eta}M^{1/2}x^{1/2}\bigg(1 + \dfrac{|a|}{xM}\bigg)^{1/2}\sum_{p\leq y}\sum_{z/p<w\leq z}\dfrac{1}{(pw)^{1/2}}\bigg(M^{1/2} + \dfrac{x}{(pw)^{3/4}}\bigg)\\
   & \ll_\eta x^{2\eta }\bigg(1+\dfrac{|a|}{xM}\bigg)^{1/2}(Mx^{1/2}y^{1/2}z^{1/2} + x^{3/2}M^{1/2}z^{-1/4}),
 \end{align*}
which, together with the relation \eqref{Kloos initial step}, gives the desired bound in the first case by adjusting $\eta$.

Let us now consider the case $(a, upw)>1$ in the analysis of our sum $K(p, w)$ given by \eqref{def K(p, w)}. Recall from \eqref{Kloos initial step} and \eqref{Introduction of K(p,w)}, we have
\begin{align*}
    \mathrm{Kl}_y(M, x; a, q)\ll \sum_{p\leq y}\sum_{z/p<w\leq z}K(p, w) + Mz.
\end{align*}

Our strategy is to reduce this case to the first case. Indeed, we write $(a,p)=r$ with $a=ra_1,\ p=rp'$ (clearly, either $p'=1,$ or $p'=p$), so that
$$\mathrm{Kl}_y(M, x; a, q)\ll \sum_{\st{a=ra_1}}\sum_{\st{p'\leq y/r\\(p', a_1)=1}}\sum_{\st{z/p'r<w\leq z}}\sum_{\st{m\sim M\\(m,p'rw)=1}}\bigg| \sum_{\st{z/p'rw\leq u<x/p'rw\\P^+(u)\leq p'r\\(u,mq)=1}}e\bigg(\dfrac{a_1\overline{up'w}}{m}\bigg)\bigg|+Mz.$$
Now, let $(a_1, w)=s$ and put $a_1=sa_2, w=sw'$ with $(a_2, w')=1$. Then 
$$\mathrm{Kl}_y(M, x; a, q)\ll \sum_{\st{a=rsa_2}}\sum_{\st{p'\leq y/r\\(p', a_2)=1}}\sum_{\st{z/p'rs<w'\leq z/s\\(w', a_2)=1}}\sum_{\st{m\sim M\\(m,p'rsw')=1}}\bigg| \sum_{\st{z/p'rsw'\leq u<x/p'rsw'\\P^+(u)\leq p'r\\(u,mq)=1}}e\bigg(\dfrac{a_2\overline{up'w'}}{m}\bigg)\bigg|+Mz.$$ 
Finally, we put $(a_2, u)=t$ and $a_2=ta_3,\ u=tu'$ with $(a_3, u')=1$. Then 
$$\mathrm{Kl}_y(M, x; a, q)\ll \sum_{\st{a=rsta_3}}\sum_{\st{p'\leq y/r\\(p', a_3)=1}}\sum_{\st{z/p'rs<w'\leq z/s\\(w',a_3)=1}}\sum_{\st{m\sim M\\(m,p'w'rs)=1}}\bigg| \sum_{\st{z/p'rsw't\leq u'<x/p'rsw't\\P^+(u')\leq p'r\\(u',a_3mq)=1}}e\bigg(\dfrac{a_3\overline{u'p'w'}}{m}\bigg)\bigg|+Mz.$$
Now, for each fix factorisation $a=rstq_3$, define 
$${K}'(p',w'):=\sum_{\st{m\sim M\\(m,p'w'rs)=1}}\bigg| \sum_{\st{z/p'rsw't\leq u'<x/p'rsw't\\P^+(u')\leq p'r\\(u',a_3mq)=1}}e\bigg(\dfrac{a_3\overline{u'p'w'}}{m}\bigg)\bigg|,$$
with $p'r=p,\ w's=w$ and $u't=u.$
Since $(a_3, p'w'u')=1$, we can apply the above coprime case similar to ${K}(p,w)$, to evaluate this, and then we sum over $p'$ and $w'$ trivially. Finally, since the number of factorisation $a=rstq_3$ is bounded by $O(|a|^\eta)$, the theorem follows by adjusting $\eta>0$ suitably.
\end{proof}

\section{Type I sums}\label{type1_sums}

\subsection{Type I sums via Vinogradov's methods}
The goal of this subsection is to establish the following Type I sum estimate.
\begin{prop}\label{Prop: Main 1}
Let $M, N\geq 2$ be two real numbers such that
\begin{align*}
    X/4\leq MN\leq 4X \quad \text{and} \quad  \dfrac{q}{R^{1-\delta}}\leq N\leq R^{2-\delta}
\end{align*}
for some $\delta>0$. Then, for any $0<\eta<\delta/2$, we have
\begin{align*}
    \sum_{m\sim M}1_{S_q(Y)}(m)\sum_{n\sim N}\Phi_a(mn, R)=\frac{\widehat{\phi}(0)NR}{q}\sum_{m\sim M}1_{S_q(Y)}(m)+O_{\delta}\lr{ R^{2-\eta}}.
\end{align*}
\end{prop}

To prove the above proposition, we will establish the following general result, which follows from Vinogradov's exponential sums method.

\begin{lem}\label{thm1_type1}
 Let $(a(m))_{m\geq 1}$ be a complex-valued sequence with $|a(m)|\leq 1$. Suppose that $M, N\geq 2$ satisfy $X/4\leq MN\leq 4X$ and $M, N\leq R^{2-\delta}$ for some $\delta>0$. Then, for any $0<\eta<\delta/2$,
    $$\sum_{n\sim N}\sum_{\st{m\sim M}}a(m)\Phi_a(mn,R)=\frac{\widehat{\phi}(0)NR}{q}\sum_{m\sim M}a(m)+O_{\delta}\lr{R^{2-\eta}}.$$
\end{lem}

\begin{proof}
    By the definition of $\Phi$ given by \eqref{def_ Phi_original} and the Poisson summation formula, we have
    \begin{align*}
        \Phi_a(mn, R)=\sum_{r\equiv amn\Mod q}\phi\bigg(\dfrac{r}{R}\bigg)=\sum_{k\in \mathbb{Z}}\phi\bigg(\dfrac{kq+amn}{R}\bigg)=\dfrac{R}{q}\sum_{k\in \mathbb{Z}}\widehat{\phi}\bigg(\dfrac{kR}{q}\bigg)e\bigg(\dfrac{amnk}{q}\bigg).
    \end{align*}
     This implies that
    \begin{align*}
        \sum_{n\sim N}\sum_{\st{m\sim M}}a(m)\p{mn,R}=\dfrac{R}{q}\sum_{k\in \mathbb{Z}}\widehat{\phi}\bigg(\dfrac{kR}{q}\bigg)\sum_{n\sim N}\sum_{m\sim M}a(m)e\bigg(\dfrac{amnk}{q}\bigg).
    \end{align*}
    Since $|\widehat{\phi}(x)|\ll_{B}\min\{1, |x|^{-B}\}$ for any positive integer $B$, we have that for $\eta^\prime>0$
    \begin{align*}
        \dfrac{R}{q}\sum_{|k|>qR^{-1+\eta^\prime}}\widehat{\phi}\bigg(\dfrac{kR}{q}\bigg)\sum_{n\sim N}\sum_{m\sim M}a(m)e\bigg(\dfrac{amnk}{q}\bigg)\ll_{\eta^\prime} \dfrac{1}{R^{2026}}.
    \end{align*}
   Consequently, we have that
    \begin{align*}
         \sum_{n\sim N}  \sum_{\st{m\sim M}}a(m)\p{mn,R} = \dfrac{\widehat{\phi}(0)NR}{q}\sum_{m\sim M}a(m) + O\bigg(\Sigma_R+ \dfrac{1}{R^{2026}}\bigg),
    \end{align*}
    where 
    \begin{align*}
        \Sigma_R:=\dfrac{R}{q}\sum_{1\leq k\leq qR^{-1+\eta^\prime}}\sum_{m\sim M}\bigg|\sum_{n\sim N}e\bigg(\dfrac{amnk}{q}\bigg)\bigg|.
    \end{align*}
    The task now reduces to showing that $\Sigma_R\ll_{\delta} R^{2-\eta}$ for any $0<\eta<\delta/2$. Indeed, by geometric series estimate, we have
    \begin{align*}
        \bigg|\sum_{n\sim N}e\bigg(\dfrac{amnk}{q}\bigg)\bigg|\ll \min\bigg(N, \bigg\|\dfrac{amk}{q}\bigg\|^{-1}\bigg).
    \end{align*}
    Next, we write $\ell=mk$. Then, by \cite[Lemma 13.7]{IwaKow2004} we infer that for any $\varepsilon>0$,
    \begin{align*}
        \Sigma_R\ll \dfrac{R^{1+\varepsilon}}{q}\sum_{1\leq \ell \ll qMR^{-1+\eta^\prime}}\min\bigg(N, \bigg\|\dfrac{a\ell}{q}\bigg\|^{-1}\bigg) &\ll \dfrac{R^{1+\varepsilon}}{q}\bigg(\dfrac{qMR^{-1+\eta^\prime}}{q} + 1\bigg)(N+q\log q)\\
        &\ll R^{1+\eta^\prime + \varepsilon} + \dfrac{NR^{1+\varepsilon}}{q} + R^{1+\varepsilon}\log q + MR^{\eta^\prime} \log q,
    \end{align*}
    using the fact that $X=qR$. Since $M, N\ll R^{2-\delta}$, choosing $\eta^\prime=\varepsilon=\delta/2$, we conclude that for any $0<\eta<\delta/2$,
    \begin{align*}
        \Sigma_R\ll_{\delta}R^{2-\eta}.
    \end{align*}
    This completes the proof.
    \end{proof}

    We are now ready to complete the proof of Proposition \ref{Prop: Main 1}.

    \begin{proof}[Proof of Proposition \ref{Prop: Main 1}] 
    By our hypothesis and the fact that $X=qR$, we have that
    \begin{align*}
        M\ll \dfrac{X}{N}\ll \dfrac{qR}{q/R^{1-\delta}}\ll R^{2-\delta}.
    \end{align*}
    Hence, the proposition follows immediately from Lemma \ref{thm1_type1} by choosing $a(m)=1_{S_q(Y)}(m)$.
        \end{proof}

        We end this subsection by establishing the following lemma.
\begin{lem}\label{thm2_type1}
Suppose that $M, N\geq 2$ satisfy $X/4\leq MN\leq 4X$ and $M, N\leq R^{2-\delta}$
    for some $\delta>0$. Then, for any $0<\eta<\delta/2$, we have
    $$\sum_{m\in \mathbb{Z}}\phi\lr{\frac{m}{3M}}\sum_{n\sim N}\Phi_a(mn,R)=\frac{3\widehat{\phi}(0)^2MNR}{q}+O_\delta(R^{2-\eta}).$$
\end{lem}
\begin{proof}
    The proof is exactly in the same way as in Lemma \ref{thm1_type1} followed by partial summation to remove the smooth weights $\phi(m/3M)$. Indeed, arguing as in the proof of Lemma \ref{thm1_type1}, for any $0<\eta<\delta/2$, we have
    \begin{align}\notag 
        \sum_{m\in \mathbb{Z}}\phi\lr{\frac{m}{3M}}\sum_{n\sim N}\Phi_a(mn,R) & =\dfrac{\widehat{\phi}(0)NR}{q}\sum_{m\in \mathbb{Z}}\phi\bigg(\dfrac{m}{3M}\bigg) + O_\delta(R^{2-\eta}).
        \end{align}
        Next, by partial summation and the fact that $NR/q\ll R^{2-\delta}$, the right-hand side of the above expression is
        \begin{align}
        \notag =\dfrac{\widehat{\phi}(0)NR}{q}\big(3M\widehat{\phi}(0) + O(1)\big) + O_\delta(R^{2-\eta}) = \frac{3\widehat{\phi}(0)^2MNR}{q}+O_\delta(R^{2-\eta}).
   \end{align}
    This completes the proof.
\end{proof}

        \subsection{Type I sums via geometry of numbers}
        The goal of this subsection is to estimate certain Type I sums using the geometry of numbers. 

For given $(a, q)=1$, let $\overline{a}$ denote the modular inverse of $a$ modulo $q$, that is, $a\overline{a}\equiv 1\Mod q$. Recall from \eqref{def_ Phi_original} that
        \begin{align}\label{def_Phi}
            \Phi_a(n, R):=\sum_{\substack{r\in \mathbb{Z}\\ n\equiv r\overline{a}\Mod q}}\phi\bigg(\dfrac{r}{R}\bigg).
        \end{align}

        \begin{re}\label{rem_phi}
    In the above definition, clearly without loss of generality, one can assume that $1\leq\overline{a}\leq q.$ 
\end{re}

        Next, for an integer $m\geq 1$, define
        \begin{align*}
            \lambda(m):=\{(j, k)\in \mathbb{Z}^2\colon jq + k\overline{a}\equiv 0\Mod m\}.
        \end{align*}
        Observe that $\lambda(m)$ is a lattice. Moreover, since $(a, q)=1$, we have that 
        \begin{align*}
            \{jq+k\overline{a}\colon j, k\in \mathbb{Z}\}=\mathbb{Z}.
        \end{align*}
        This implies that $jq+k\overline{a}$ represents all congruence classes modulo $m$. Hence, the determinant of the lattice $\lambda(m)$ is $m$.

        Throughout this section, we let $v_1(m)$ be the shortest non-zero vector in $\lambda(m)$. Also, let $R_1(m)$ be the Euclidean length of $v_1(m)$.

        \begin{lem}\label{thm3_type1}
Suppose that $M, N\geq 2$ satisfy $X/4\leq MN\leq 4X$ and $M, N\leq R^{2-\delta}$
    for some $\delta>0$. Then, for any $0<\eta <\delta/4$, we have
    $$\sum_{m\in \mathbb{Z}}\phi\lr{\frac{m}{3M}}\sum_{n_1\sim N}\sum_{n_2\sim N}\Phi_a(mn_1,R)\Phi_a(mn_2,R)=\frac{3\widehat{\phi}(0)^3MN^2R^2}{q^2}+O_{\delta}\lr{\frac{R^{4-\eta}}{M}}.$$
\end{lem}

Irving has established the above result in \cite[Theorem 4.8]{Irv2014}. However, for completeness, we provide a brief proof below.

\begin{proof}
    For any $m\asymp M$, by \cite[Lemma 4.2]{Irv2014}, we have
    \begin{align*}
        \sum_{n\sim N}\Phi_a(mn, R)=\dfrac{\widehat{\phi}(0)NR}{q} + O\bigg(\dfrac{R}{R_1(m)}\bigg).
    \end{align*}
    Therefore, we have
    \begin{align*}
        \sum_{m\in \mathbb{Z}} &\phi\lr{\frac{m}{3M}}\sum_{n_1\sim N}\sum_{n_2\sim N}\Phi_a(mn_1,R)\Phi_a(mn_2,R)\\
        &=\dfrac{\widehat{\phi}(0)NR}{q}\sum_{m\in \mathbb{Z}}\phi\bigg(\dfrac{m}{3M}\bigg)\sum_{n\sim N}\Phi_a(mn, R) + O\bigg(R\sum_{m\in \mathbb{Z}}\phi(m/3M)R_1(m)^{-1}\sum_{n\sim N}\Phi_a(mn, R)\bigg).
    \end{align*}
    Invoking Lemma \ref{thm2_type1} together with the fact that $MN\asymp X=qR$, we see that for any $0<\eta<\delta/2$, the first term in the right-hand side of the above expression is 
    \begin{align*}
       = \dfrac{3\widehat{\phi}(0)^3MN^2R^2}{q^2} + O_{\delta}\bigg(\dfrac{R^{4-\eta}}{M}\bigg).
    \end{align*}
    Next, by the Cauchy-Schwarz inequality, we have
    \begin{align*}
        \bigg|R\sum_{m\in \mathbb{Z}}\phi(m/3M)R_1(m)^{-1}\sum_{n\sim N}\Phi_a(mn, R)\bigg|^2\ll \bigg(\sum_{m\asymp M}\bigg(\sum_{n\sim N}\Phi_a(mn, R)\bigg)^2\bigg)\bigg(\sum_{m\asymp M}\dfrac{R^2}{R_1(m)^2}\bigg).
    \end{align*}
    We may now apply \cite[Lemmas 4.6, 4.7]{Irv 2014} to conclude for any $\varepsilon>0$, the right-hand side of the above expression is $\ll_{\varepsilon} (NR^3/q)\cdot R^{2+\varepsilon}$, so that
    \begin{align*}
      R\sum_{m\in \mathbb{Z}}\phi(m/3M)R_1(m)^{-1}\sum_{n\sim N}\Phi_a(mn, R) \ll_{\varepsilon} \dfrac{N^{1/2}R^{5/2+\varepsilon/2}}{q^{1/2}}\ll_{\varepsilon}\dfrac{M^{1/2}R^{3+\varepsilon/2}}{M}\ll \dfrac{R^{4-\delta/2+\varepsilon/2}}{M},
    \end{align*}
    using the facts that $X=qR$, $MN\asymp X$ and $M\ll R^{2-\delta}$. Choosing $\varepsilon=\delta/4$ completes the proof of the lemma.
\end{proof}

        \section{Type II sums}\label{type2_sums}

        The goal of this section is to establish the following Type II estimate.

        \begin{prop}\label{Prop: Main 2}
Let $M, N\geq 2$ be two real numbers such that 
\begin{align}\label{con_range1}
    X/4\leq MN\leq 4X\quad \text{and} \quad \dfrac{q}{R^{1-\delta}}\leq N\leq R^{12/11-\delta}
\end{align} 
for some $\delta>0$. Let $Y=(\log X)^{C(\delta)}$ for some large $C(\delta)>0$. Then, for any $0<\eta<\delta/20$, we have
\begin{align*}
    \sum_{m\sim M}1_{S_q(Y)}(m)\sum_{n\sim N}\big(1_{S_q(Y)}(n)-K(N,Y)\big)\Phi_a(mn, R)\ll_{\delta} R^{2-\eta},
\end{align*}
where 
\begin{align}\label{def_K(x,y)}
    K(x,y):=\frac{1}{x}\sum_{n\sim x}1_{S_q(y)}(n).
\end{align}
    \end{prop}

    The proof of Proposition ~\ref{Prop: Main 2} will be divided into several steps. For brevity, we denote
\begin{align*}
    S:=\sum_{m\sim M}1_{S_q(Y)}(m)\sum_{n\sim N}\big(1_{S_q(Y)}(n)-K(N,Y)\big)\Phi_a(mn, R),
\end{align*}
where $K(x,y)$ given by \eqref{def_K(x,y)}. The goal is to show that $S\ll_\delta R^{2-\eta}$ for some $\eta>0$ depending on $\delta$. Throughout this section, we assume that $Y=(\log X)^{C(\delta)}$ for some large $C(\delta)>0$.

\subsection{Analysis of the sum $S$ via dispersion method}
By the Cauchy-Schwarz inequality and the definition of $\phi$, we have
\begin{align*}
S^2
&\ll M\sum_{m}\phi\bigg(\dfrac{m}{3M}\bigg)\bigg|\sum_{n\sim N}\big(1_{S_q(Y)}(n)-K(N, Y)\big)\Phi_a(mn, R)\bigg|^2\\
    &= M(S'_1-2S'_2 +S'_3),
\end{align*}
where
\begin{align}\label{def_S11}
    S'_{1}=\sum_{n_1\sim N}\sum_{n_2\sim N}1_{S_q(Y)}(n_1)1_{S_q(Y)}(n_2)\sum_{m}\phi\lr{\frac{m}{3M}}\p{mn_1,R}\p{mn_2,R},
\end{align}
\begin{align}\label{def_S12}
    S'_{2}=K(N,Y)\sum_{n_1\sim N}\sum_{n_2\sim N}1_{S_q(Y)}(n_1)\sum_{m}\phi\lr{\frac{m}{3M}}\p{mn_1,R}\p{mn_2,R},
\end{align} and 
\begin{align}\label{def_S13}
    S'_{3}=K^2(N,Y)\sum_{n_1\sim N}\sum_{n_2\sim N}\sum_{m}\phi\lr{\frac{m}{3M}}\p{mn_1,R}\p{mn_2,R}.
\end{align}
Next, for notational convenience, we write
\begin{align}\label{def_S'_new}
     S':=S'_{1}-2S'_{2}+S'_{3}.
\end{align}
Thus, to prove Proposition~\ref{Prop: Main 2}, it suffices to establish the following result.

\begin{prop}\label{thm2_type2}
    Let $S'$ be as in \eqref{def_S'_new}. Suppose that $M,N\geq 2$ satisfy the relation \eqref{con_range1}. Then, for any $0<\eta<\delta/20$, we have 
    \begin{align*}
        S'\ll_{\delta}\frac{R^{4-2\eta}}{M}.
    \end{align*}
\end{prop}

We begin by estimating the sum $S_3^\prime$.

\begin{lem}\label{lem_S3'_new}
   Let $S'_{3}$ be as in \eqref{def_S13}. Suppose that $M,N\geq 2$ satisfy the relation \eqref{con_range1}. Then, for any $0<\eta<\delta/20$, we have
    \begin{align*}
        S'_{3}=\frac{3\widehat{\phi}(0)^3K^2(N,Y)MN^2R^2}{q^2}+O_{\delta}\lr{\frac{K^2(N,Y)R^{4-2\eta}}{M}}.
    \end{align*}
\end{lem}

\begin{proof}
    The proof follows from Lemma \ref{thm3_type1} by noting that \eqref{con_range1} implies $M, N\leq R^{2-\delta}$.
\end{proof}

Next, to estimate the sums $S'_1$ and $S'_2$, we will prove the following general result for any $1$-bounded sequence. To be precise, let  $\boldsymbol{\beta}=(\beta (n))_{n\geq 1}$ be a complex-valued sequence with $|\beta(n)|\leq 1$. Define
\begin{align}\label{def_S_beta}
        S(\boldsymbol{\beta}):=\sum_{n_1\sim N}1_{S_q(Y)}(n_1)\sum_{n_2\sim N}\beta(n_2)\sum_m \phi\bigg(\dfrac{m}{3M}\bigg)\Phi_a(mn_1, R)\Phi_a(mn_2, R).
    \end{align}

\begin{prop}\label{Prop: combine S1 and S2}
    Suppose that  $M,N\geq 2$ satisfy the relation \eqref{con_range1}. Let  $\boldsymbol{\beta}=(\beta (n))_{n\geq 1}$ be a complex-valued sequence with $|\beta(n)|\leq 1$. Then,  for any $0<\eta<\delta/20$, we have
    \begin{align*}
        S(\boldsymbol{\beta})=\dfrac{3\widehat{\phi}(0)^3K(N, Y)MNR^2}{q^2}\sum_{n\sim N}\beta(n) + O_\delta\bigg(\dfrac{R^{4-2\eta}}{M}\bigg).
    \end{align*}
\end{prop}

Throughout the remainder of the section, our goal is to establish Proposition \ref{Prop: combine S1 and S2}. Before that, we quickly demonstrate how to deduce Proposition \ref{thm2_type2} from Proposition \ref{Prop: combine S1 and S2}.

\begin{proof}[Proof of Proposition \ref{thm2_type2} assuming Proposition \ref{Prop: combine S1 and S2}]
Applying Proposition \ref{Prop: combine S1 and S2} with $\beta(n)=1_{S_q(Y)}(n)$ and $\beta(n)=1$, we have that
\begin{align*}
    S'_1=\frac{3\widehat{\phi}(0)^3K^2(N,Y)MN^2R^2}{q^2}+O_{\delta}\lr{\frac{R^{4-2\eta}}{M}}
\end{align*}
and
\begin{align*}
     S'_2=\frac{3\widehat{\phi}(0)^3K^2(N,Y)MN^2R^2}{q^2}+O_{\delta}\lr{\frac{K(N,Y)R^{4-2\eta}}{M}},
\end{align*}
respectively. Finally, invoking Lemma~\ref{lem_S3'_new} and noting that $K(N, Y)\leq 1$, the relation \eqref{def_S'_new} implies that
\begin{align*}
    S'=O_{\delta}\lr{\frac{R^{4-2\eta}}{M}},
\end{align*}
as desired.
\end{proof}

Therefore, our goal reduces to establishing Proposition \ref{Prop: combine S1 and S2}.

\subsection{Initial analysis of the sum $S(\boldsymbol{\beta})$} We begin by noting that, if  $M, N\geq 2$ satisfy the relation \eqref{con_range1}, then 
\begin{align*}
    M\ll \dfrac{X}{N}\ll \dfrac{qR}{q/R^{1-\delta}}\ll R^{2-\delta},
\end{align*}
using the fact that $X=qR$. In particular, we can apply the estimates from Section~\ref{type1_sums}, which we shall do throughout this section. 

We begin by dividing the sum $S(\boldsymbol{\beta})$ into two parts according to the greatest common divisor of $n_1$ and $n_2$. More precisely, for some $1\leq D_0\leq 2N$, we write
\begin{align}\label{S_1'=S1+S2}
    S(\boldsymbol{\beta})=\sum_{1\leq d<D_0}S(\boldsymbol{\beta}, d) + \sum_{D_0\leq d\leq 2N}S(\boldsymbol{\beta}, d) :=S_1(\boldsymbol{\beta}) + S_2(\boldsymbol{\beta}),
\end{align}
say. Here and throughout this section,
\begin{align}\label{def_S2(d)}
    S(\boldsymbol{\beta}, d):=\sum_{\st{n_1, n_2\sim N/d\\(n_1,n_2)=1}}1_{S_q(Y)}(dn_1)\beta(dn_2)\sum_{m}\phi\lr{\frac{m}{3M}}\p{mdn_1,R}\p{mdn_2,R}.
\end{align}

We take 
\begin{align}\label{def_D0}
    D_0=R^{2\eta},
\end{align}
where $0<\eta<\delta/20$.

We begin by estimating the contributions from the sum $S_2(\boldsymbol{\beta})$ in the following lemma.

\begin{lem}\label{lem_S2_new}
    Let $S_2(\boldsymbol{\beta})$ be as given in \eqref{S_1'=S1+S2}. Suppose that the condition \eqref{con_range1} holds. Then, for any $0<\eta<\delta/20$, we have
    $$S_2(\boldsymbol{\beta})\ll_{\delta} \frac{R^{4-2\eta}}{M}.$$
\end{lem}

\begin{proof} By definition, we have
    \begin{align*}
        S_2(\boldsymbol{\beta})=\sum_{D_0\leq d\leq 2N}\sum_{\st{n_1, n_2\sim N/d\\(n_1,n_2)=1}}1_{S_q(Y)}(dn_1)\beta(dn_2)\sum_{m}\phi\lr{\frac{m}{3M}}\p{mdn_1,R}\p{mdn_2,R}.
    \end{align*}
    Note that the above sum contributes only when $(d, q)=1$ as $dn_1$ is supported on integers coprime to $q$. Since $(a, q)=1$ and $d\leq 2N\leq R^{2-\delta}<q$, we have that $da\equiv a'\Mod q$ for some $(a', q)=1$ and $1\leq a'\leq q$. So, by definition \eqref{def_Phi} of $\Phi$ and Remark \ref{rem_phi}, we have $\Phi(mdn_i, R)=\Phi_{a'}(mn_i, R)$ for $i\in \{1, 2\}$. So, we may now apply Lemma \ref{thm3_type1} to obtain
    \begin{align*}
        S_2(\boldsymbol{\beta}) &\leq \sum_{D_0\leq d\leq 2N}\sum_{n_1, n_2\sim N/d}\sum_{m}\phi\bigg(\dfrac{m}{3M}\bigg)\Phi_{a'}(mn_1, R)\Phi_{a'}(mn_2, R)\\
        &\ll \sum_{D_0\leq d\leq 2N}\dfrac{MN^2R^2}{d^2q^2}\ll\dfrac{MN^2R^2}{q^2D_0}\ll \dfrac{R^{4-2\eta}}{M}, 
    \end{align*}
    using the facts that $MN\asymp X=qR$ and $D_0=R^{2\eta}$. This completes the proof.
\end{proof}

So, it remains to evaluate the sum $S_1(\boldsymbol{\beta})$, which constitutes the remaining bulk of this section.

\subsection{Fourier analysis on the sum $S_1(\boldsymbol{\beta})$} 

Recall that
\begin{align*}
    S_1(\boldsymbol{\beta})=\sum_{1\leq d<D_0}\sum_{\st{n_1, n_2\sim N/d\\(n_1,\: n_2)=1}}1_{S_q(Y)}(dn_1)\beta(dn_2)\sum_{m}\phi\lr{\frac{m}{3M}}\p{mdn_1,R}\p{mdn_2,R}.
\end{align*}
We will apply Fourier analysis in order to estimate the above sum. We write
\begin{align*}
    S_1(\boldsymbol{\beta})=\sum_{1\leq d<D_0}\sum_{\st{n_1, n_2\sim N/d\\(n_1,n_2)=1}}1_{S_q(Y)}(dn_1)\beta(dn_2)F(n_1, n_2; d),
\end{align*}
where
\begin{align*}
    F(n_1, n_2, d):=\sum_{m}\phi\lr{\frac{m}{3M}}\p{mdn_1,R}\p{mdn_2,R}.
\end{align*}
By the definition of $\Phi$ given by \eqref{def_Phi} and then applying the Poisson summation formula three times (see, for example, Lemma \cite[Lemma 5.4]{Irv2014}), we may express the above sum $F(n_1, n_2, d)$ as
\begin{align*}
    F(n_1,n_2,d)=\frac{3MR^2}{q^2}\sum_{k_1}\sum_{k_2}\widehat{\phi}\lr{\frac{k_1R}{q}}\widehat{\phi}\lr{\frac{k_2R}{q}}\sum_{m}\widehat{\phi}\lr{3Mm-\frac{3Mad(k_1n_1+k_2n_2)}{q}}.
\end{align*}
The next step is to analyse $F(n_1, n_2, d)$ in two parts, depending on whether $k_1n_1+k_2n_2=0$ or not. For ease of notation, here and throughout this section, we write
\begin{align}\label{def_F1}
    F_1(n_1,n_2,d):=\frac{3MR^2}{q^2}\sum_{\st{k_1,k_2\\k_1n_1+k_2n_2=0}}\widehat{\phi}\lr{\frac{k_1R}{q}}\widehat{\phi}\lr{\frac{k_2R}{q}}\sum_{m}\widehat{\phi}\lr{3Mm}
\end{align}
and 
\begin{align}\label{def_F2}
    F_2(n_1,n_2,d):=\frac{3MR^2}{q^2}\sum_{\st{k_1,k_2\\k_1n_1+k_2n_2\neq 0}}\widehat{\phi}\lr{\frac{k_1R}{q}}\widehat{\phi}\lr{\frac{k_2R}{q}}\sum_{m}\widehat{\phi}\lr{3Mm-\frac{3Mad(k_1n_1+k_2n_2)}{q}}.
\end{align}
Furthermore, we write
\begin{align}\label{def2_S1}
    S_1(\boldsymbol{\beta})=S_3(\boldsymbol{\beta})+S_4(\boldsymbol{\beta}),
\end{align}
where 
\begin{align}\label{def_S3_new}
    S_3(\boldsymbol{\beta})=\sum_{d<D_0}\sum_{\st{n_1, n_2\sim N/d\\(n_1,n_2)=1}}1_{S_q(Y)}(dn_1)\beta(dn_2)F_1(n_1,n_2,d)
\end{align}
and 
\begin{align}\label{def_S4_new}
    S_4(\boldsymbol{\beta})=\sum_{d<D_0}\sum_{\st{n_1, n_2\sim N/d\\(n_1,n_2)=1}}1_{S_q(Y)}(dn_1)\beta(dn_2)F_2(n_1,n_2,d).
\end{align}

\subsection{Estimation of the sum $S_3(\boldsymbol{\beta})$}
We begin by evaluating the sum $S_3(\boldsymbol{\beta})$.

\begin{lem}\label{lem_S3_new}
    Suppose that the condition \eqref{con_range1} holds. Let  $\boldsymbol{\beta}=(\beta (n))_{n\geq 1}$ be a complex-valued sequence with $|\beta(n)|\leq 1$. Let $S_3(\boldsymbol{\beta})$ be as in \eqref{def_S3_new}.  Then, for any $0<\eta<\delta/20$, we have
    \begin{align*}
        S_3(\boldsymbol{\beta})=\dfrac{3\widehat{\phi}(0)^3K(N, Y)MNR^2}{q^2}\sum_{n\sim N}\beta(n) + O_\delta\bigg(\dfrac{R^{4-2\eta}}{M}\bigg).
        \end{align*}
    \end{lem}

    \begin{proof}
        First observe that if $(n_1, n_2)=1$ and $k_1n_1+k_2n_2=0$, then we may write $k_1=hn_2$ and $k_2=-hn_1$ for some integer $h$. So, by \eqref{def_S3_new} and \eqref{def_F1}, we can express the sum $S_3(\boldsymbol{\beta})$ as
        \begin{align*}
            S_3(\boldsymbol{\beta})=\dfrac{3MR^2}{q^2}\sum_{d<D_0}\sum_{\st{n_1, n_2\sim N/d\\(n_1,n_2)=1}}1_{S_q(Y)}(dn_1)\beta(dn_2)\sum_{h}\widehat{\phi}\bigg(\dfrac{-n_1hR}{q}\bigg)\widehat{\phi}\bigg(\dfrac{n_2hR}{q}\bigg)\sum_{m}\widehat{\phi}(3Mm).
        \end{align*}
    Next, we note that the condition \eqref{con_range1} implies that 
    \begin{align*}
        \dfrac{n_1R}{q}, \dfrac{n_2R}{q}\gg \dfrac{NR}{dq}\gg R^{\delta-2\eta}  \quad \text{and} \quad M \gg \dfrac{X}{N}\gg \dfrac{qR}{R^{2-\delta}}\gg R^\delta.
    \end{align*}
    Since $\eta<\delta/20$, using the fact that $|\widehat{\phi}(x)|\ll_{B}\min\{1, |x|^{-B}\}$ for any positive integer $B$, we can bound the contributions of the terms with $h\neq 0$ or $m\neq 0$ in $S_3(\boldsymbol{\beta})$ by $\ll R^{-2026}$. So, we can deduce that
    \begin{align}\label{Eq: S3 beta}
         S_3(\boldsymbol{\beta})=\frac{3\widehat{\phi}(0)^3MR^2}{q^2}\sum_{1\leq d< D_0}\sum_{\st{n_1, n_2\sim N/d\\(n_1,n_2)=1}}1_{S_q(Y)}(dn_1)\beta(dn_2)+O(R^{-2026}).
    \end{align}
    Observe that
    \begin{align*}
        \dfrac{MR^2}{q^2}\sum_{D_0\le d\leq 2N}\sum_{\st{n_1, n_2\sim N/d\\(n_1,n_2)=1}}1_{S_q(Y)}(dn_1)\beta(dn_2)\ll  \dfrac{MR^2}{q^2}\sum_{D_0\le d\leq 2N} \dfrac{N^2}{d^2}\ll_{\delta} \dfrac{R^{4-2\eta}}{M},
    \end{align*}
    using the facts that $MN\asymp X=qR$ and $D_0=R^{2\eta}$. So, we may extend the sum over $d$ in \eqref{Eq: S3 beta} to the full interval $1\leq d\leq 2N$ to conclude that 
    \begin{align*}
        S_3(\boldsymbol{\beta})=\frac{3\widehat{\phi}(0)^3MR^2}{q^2}\sum_{n_1, n_2\sim N}1_{S_q(Y)}(n_1)\beta(n_2) + O_\delta\bigg(\dfrac{R^{4-2\eta}}{M}\bigg).
    \end{align*}
    Recalling that $K(N, Y)=N^{-1}\sum_{n\sim N}1_{S_q(Y)}(n)$ completes the proof.
    \end{proof}

\subsection{Initial analysis of the sum $S_4(\boldsymbol{\beta})$}
We now focus on the sum $S_4(\boldsymbol{\beta})$ given by
 \[S_4(\boldsymbol{\beta})=\sum_{1\leq d<D_0}\sum_{\st{n_1, n_2\sim N/d\\(n_1,n_2)=1}}1_{S_q(Y)}(dn_1)\beta(dn_2)F_2(n_1,n_2,d),\]
 where $F_2(n_1, n_2, d)$ is given by \eqref{def_F2}. 

 First, we note that for any given $d, m, k_1, k_2$, there exists a unique integer $u$ such that
 \begin{align*}
     m-\dfrac{ad(k_1n_1 + k_2n_2)}{q}=\dfrac{u}{q}.
 \end{align*}
So, we must have a unique integer $v$ such that
\begin{align}\label{def_c}
    k_1n_1+k_2n_2=vq-a''u=c,
\end{align}
say. Here $a''\in [1, q]$ is an integer such that $(da)a'' \equiv 1\Mod q$. Note that $a''$ exists as $(da, q)=1$. So, we may express $F_2(n_1, n_2, d)$ as
\begin{align}\label{F2}
    F_2(n_1, n_2, d)=\dfrac{3MR^2}{q^2}\sum_{\substack{k_1, k_2, u, v \in \mathbb{Z}\\ k_1n_1+k_2n_2=c\neq 0}}\widehat{\phi}\bigg(\dfrac{k_1R}{q}\bigg)\widehat{\phi}\bigg(\dfrac{k_2R}{q}\bigg)\widehat{\phi}\bigg(\dfrac{3Mu}{q}\bigg).
\end{align}
Next, if $(n_1, n_2)=1$, we may apply \cite[Lemma 5.6]{Irv2014} to obtain
\begin{align*}
    \sum_{\substack{k_1, k_2 \in \mathbb{Z}\\ k_1n_1+k_2n_2=c\neq 0}}\widehat{\phi}\bigg(\dfrac{k_1R}{q}\bigg)\widehat{\phi}\bigg(\dfrac{k_2R}{q}\bigg)=\dfrac{1}{n_2}\sum_{w\in \mathbb{Z}}\widehat{g}(w/n_2)e\bigg(\dfrac{c\overline{n_1}w}{n_2}\bigg),
\end{align*}
where $\overline{n_1}$ denote the inverse of $n_1$ modulo $n_2$ and for any real number $t$,
\begin{align}\label{def_g}
       g(t):= g(t;n_1, n_2, c)=\widehat{\phi}\lr{\frac{tR}{q}}\widehat{\phi}\lr{\frac{(c-tn_1)R}{n_2q}}.
    \end{align}
So, from the above discussion, we can infer that
\begin{align*}
    S_4(\boldsymbol{\beta})=\dfrac{3MR^2}{q^2}\sum_{1\leq d<D_0}\sum_{\substack{n_1, n_2\sim N/d\\ (n_1, n_2)=1}}1_{S_q(Y)}(dn_1)\beta(dn_2) \cdot \dfrac{1}{n_2}\sum_{\substack{u,v, w\in \mathbb{Z}\\ c\neq 0}}\widehat{\phi}\bigg(\dfrac{3Mu}{q}\bigg)\widehat{g}\bigg(\dfrac{w}{n_2}\bigg)e\bigg(\dfrac{c\overline{n_1}w}{n_2}\bigg).
\end{align*}

Next, we divide $S_4(\boldsymbol{\beta})$ into two parts depending on whether $w=0$ or $w\neq 0$ in the inner sum. Indeed, we write
\begin{align}\label{S4=S5+S6}
    S_4(\boldsymbol{\beta})=S_5(\boldsymbol{\beta})+S_6(\boldsymbol{\beta}),
\end{align}
where
\begin{align}\label{def_S5_new}
   S_5(\boldsymbol{\beta})= \frac{3MR^2}{q^2}\sum_{d< D_0}\sum_{\st{n_1,n_2\sim N/d\\(n_1, n_2)=1}}\frac{1_{S_q(Y)}(dn_1)\beta(dn_2)}{n_2}\sum_{\st{u,v \in \mathbb{Z}\\c\neq 0}}\widehat{\phi}\lr{\frac{3Mu}{q}}\widehat{g}\lr{0;n_1, n_2, c}
\end{align}
and 
\begin{align}\label{def_S6_new}
    S_6(\boldsymbol{\beta})=\frac{3MR^2}{q^2}\sum_{d< D_0}\sum_{\st{n_1,n_2\sim N/d\\(n_1, n_2)=1}}\frac{1_{S_q(Y)}(dn_1)\beta(dn_2)}{n_2}\hspace*{-0.3cm}\sum_{\st{u,v, w \in \mathbb{Z}\\ c, w\neq 0}}\hspace*{-0.3cm}\widehat{\phi}\lr{\frac{3Mu}{q}}\widehat{g}\lr{w/n_2;n_1, n_2, c}e\lr{\frac{c\overline{n_1} w}{n_2}}.
\end{align}

\subsection{Evaluation of the sum $S_5(\boldsymbol{\beta})$}
To estimate the sum $S_5(\boldsymbol{\beta})$, first we add back the term $c=0$ and then subtract it. Recall from \eqref{def_c}, we have $vq-a''u=c$. So, if $c=0$, then we must have $u=qk$ and $v=a''k$ for some integer $k$ as by our choice $(a'', q)=1$. This allows us to write
\begin{align}\label{def2_S5}
    S_5(\boldsymbol{\beta})=S_5'(\boldsymbol{\beta})-S_5''(\boldsymbol{\beta}),
\end{align}
where
\begin{align}\label{def_S5'_new}
    S_5'(\boldsymbol{\beta})=\frac{3MR^2}{q^2}\sum_{d< D_0}\sum_{\st{n_1,n_2\sim N/d\\(n_1, n_2)=1}}\frac{1_{S_q(Y)}(dn_1)\beta(dn_2)}{n_2}\sum_{\st{u,v\in \mathbb{Z}}}\widehat{\phi}\lr{\frac{3Mu}{q}}\widehat{g}\lr{0;n_1, n_2, c}
\end{align}
and 
\begin{align}\label{def_5''_new}
    S_5''(\boldsymbol{\beta})=  \frac{3MR^2}{q^2}\sum_{d< D_0}\sum_{\st{n_1,n_2\sim N/d\\(n_1, n_2)=1}}\frac{1_{S_q(Y)}(dn_1)\beta(dn_2)}{n_2}\widehat{g}\lr{0;n_1, n_2, 0}\sum_{\st{k\in \mathbb{Z}}}\widehat{\phi}\lr{3Mk}.
\end{align}

We first show that the sum $S_5'(\boldsymbol{\beta})=0$.

\begin{lem}\label{lem_S5'_new}
    Let $S_5'(\boldsymbol{\beta})$ be as in \eqref{def_S5'_new}. Suppose that the condition \eqref{con_range1} holds. Then, we have $S_5'(\boldsymbol{\beta})=0.$
\end{lem}

\begin{proof}
    By definition of $g$ from \eqref{def_g}, we have
    \begin{equation}
\begin{aligned}
    S_5'(\boldsymbol{\beta})= &\: \frac{3MR^2}{q^2}\sum_{d< D_0}\sum_{\st{n_1,n_2\sim N/d\\(n_1, n_2)=1}}\frac{1_{S_q(Y)}(dn_1)\beta(dn_2)}{n_2}\sum_{\st{ u\in \mathbb{Z}}}\widehat{\phi}\lr{\frac{3Mu}{q}}\\
    & \times \int_{-\infty}^{\infty}\widehat{\phi}\lr{\frac{tR}{q}}\sum_{v\in \mathbb{Z}}\widehat{\phi}\lr{\frac{(c-tn_1)R}{n_2q}}dt.
\end{aligned}
 \label{def2_S5'}
\end{equation}
Since $c=vq-a''u$, we can bound the sum over $v$ as
\begin{align*}
    \sum_{v\in \mathbb{Z}}\widehat{\phi}\lr{\frac{(c-tn_1)R}{n_2q}}=\sum_{v\in \mathbb{Z}}\widehat{\phi}\lr{\frac{(vq-a''u-tn_1)R}{n_2q}}=\dfrac{n_2}{R}\sum_{r\in \mathbb{Z}}\phi\bigg(\dfrac{n_2r}{R}\bigg)e\bigg(\dfrac{a''ur+tn_1r}{q}\bigg),
\end{align*}
where we have used the Poisson summation formula in the last step. Next, note that $n_2r/R=0$ if $r=0$ and by the condition \eqref{con_range1}, 
\begin{align*}
    \bigg|\dfrac{n_2r}{R}\bigg|\geq \dfrac{N}{dR}\geq \dfrac{N}{R^{1+2\eta}}\geq 1 \quad \text{if $r\neq 0$}.
\end{align*}
as $N\geq q/R^{1-\delta}$ and $\theta>1/3$. This implies that $\phi(n_2r/R)=0$ for all integers $r$ as $\phi$ is supported on $[1/4, 3/4]$. So, from \eqref{def2_S5'}, we can deduce that $S_5'(\boldsymbol{\beta})=0$.
\end{proof}

Now we turn our attention to estimating the sum $S_5''(\boldsymbol{\beta})$.

\begin{lem}\label{lem_S5''_new}
    Let $S_5''(\boldsymbol{\beta})$ be as in \eqref{def_5''_new}. Suppose that the condition \eqref{con_range1} holds. Then, for any $0<\eta<\delta/20$, we have
    \begin{align*}
        S_5''(\boldsymbol{\beta}) \ll_{\delta}\frac{R^{4-2\eta}}{M}.
    \end{align*}
\end{lem}

\begin{proof}
Our goal is to show that
\[S_5''(\boldsymbol{\beta}):=  \frac{3MR^2}{q^2}\sum_{d< D_0}\sum_{\st{n_1,n_2\sim N/d\\(n_1, n_2)=1}}\frac{1_{S_q(Y)}(dn_1)\beta(dn_2)}{n_2}\widehat{g}\lr{0;n_1, n_2, 0}\sum_{\st{k\in \mathbb{Z}}}\widehat{\phi}\lr{3Mk}\ll_\delta \dfrac{R^{4-2\eta}}{M}.\]
We first note that the condition \eqref{con_range1} and the fact that $X=qR$ imply 
    \begin{align*}
        M\gg \dfrac{X}{N}\gg \dfrac{qR}{q/R^{1-\delta}}\gg R^\delta.
    \end{align*}
Therefore, if $k\neq 0$, since $\widehat{\phi}(x)\ll_A\min \{1, |x|^{-A}\}$ for any positive integer $A$, we can bound the contributions of the terms with $k\neq0$ by $\ll R^{-2026}$. This implies that
\begin{align*}
  S _5''(\boldsymbol{\beta})= \frac{3\widehat{\phi}(0)MR^2}{q^2}\sum_{d< D_0}\sum_{\st{n_1,n_2\sim N/d\\(n_1, n_2)=1}}\frac{1_{S_q(Y)}(dn_1)\beta(dn_2)}{n_2}\widehat{g}\lr{0;n_1, n_2, 0} + O(R^{-2026}).
\end{align*}
By \cite[Lemma 5.7]{Irv2014}, we have $\widehat{g}(0; n_1, n_2, 0)\ll q/R$ for all $n_1, n_2\sim N/d$. So, we can bound the above sum trivially as
\begin{align*}
    S _5''(\boldsymbol{\beta}) \ll \dfrac{MR^2}{q^2}\sum_{d<D_0}\dfrac{N}{d}\cdot \dfrac{q}{R} + R^{-2026}\ll \dfrac{MNR\log D_0}{q} + R^{-2026}\ll R^2\log D_0,
\end{align*}
using the fact that $MN\asymp X=qR$. Since the condition \eqref{con_range1} implies that $M\ll R^{2-\delta}$, we can conclude that for any $0<\eta<\delta/20$,
\begin{align*}
     S _5''(\boldsymbol{\beta})\ll_\delta \dfrac{R^{4-2\eta}}{M},
\end{align*}
as desired.
\end{proof}

Now we are ready to complete the evaluation of the sum $S_5(\boldsymbol{\beta})$.

\begin{lem}\label{lem_for_S5}
    Let $S_5(\boldsymbol{\beta})$ be as in \eqref{def_S5_new}. Suppose that the condition \eqref{con_range1} holds. Then, for any $0<\eta<\delta/20$, we have
    \begin{align*}
        S_5(\boldsymbol{\beta}) \ll_{\delta}\frac{R^{4-2\eta}}{M}.
    \end{align*}
\end{lem}

\begin{proof}
    The proof follows immediately from the relation \eqref{def2_S5} by invoking Lemma~\ref{lem_S5'_new} and Lemma~\ref{lem_S5''_new}.
\end{proof}

\subsection{Analysis of the sum $S_6(\boldsymbol{\beta})$} 

Recall from \eqref{def_S6_new}, $S_6(\boldsymbol{\beta})$ is given by 
\begin{align}\notag 
    S_6(\boldsymbol{\beta})=\frac{3MR^2}{q^2}\sum_{d< D_0}\sum_{\st{n_1,n_2\sim N/d\\(n_1, n_2)=1}}\frac{1_{S_q(Y)}(dn_1)\beta(dn_2)}{n_2}\hspace*{-0.3cm}\sum_{\st{u,v, w \in \mathbb{Z}\\ c, w\neq 0}}\hspace*{-0.3cm}\widehat{\phi}\lr{\frac{3Mu}{q}}\widehat{g}\lr{w/n_2;n_1, n_2, c}e\lr{\frac{c\overline{n_1} w}{n_2}}.
\end{align}
We will bound $S_6(\boldsymbol{\beta})$ by dividing it into two parts. In the first part, using the decay of the function $\phi$, we prove that we can restrict the ranges of the variables $u$, $v$, and $w$. Then, in the second part, using those restrictions together with the Kloosterman sums bounds, we prove that $S_6(\boldsymbol{\beta})$ contributes to the error term. More precisely, for $0<\eta<\delta/20$, let us define the set $\mathcal{A}$ given by
\begin{align}\label{def_A}
     \mc{A}:=\bigg\{(u, v, w)\in\mb{Z}^3: |u|<qR^{\eta}/M, |v|< N/R^{1-4\eta}, |w|< 8NR/q\bigg\}.
 \end{align}
Next, we write
\begin{align}\label{def2_S6}
    S_6(\boldsymbol{\beta})=S_7(\boldsymbol{\beta})+S_8(\boldsymbol{\beta}),
\end{align}
where
\begin{align}\label{def_S7_new}
    S_7(\boldsymbol{\beta})=\frac{3MR^2}{q^2}\sum_{d< D_0}\sum_{\st{n_1,n_2\sim N/d\\(n_1, n_2)=1}}\hspace*{-0.4cm}\frac{1_{S_q(Y)}(dn_1)\beta(dn_2)}{n_2}\hspace*{-0.4cm}\sum_{\st{(u,v,w)\in\mb{Z}^3\setminus\mc{A}\\w\neq 0, c\neq 0}}\widehat{\phi}\lr{\frac{3Mu}{q}}\widehat{g}\lr{w/n_2;n_1, n_2, c}e\lr{\frac{c\overline{n_1} w}{n_2}}
\end{align}
and 
\begin{align}\label{def_S8_new}
    S_8(\boldsymbol{\beta})=\frac{3MR^2}{q^2}\sum_{d< D_0}\sum_{\st{n_1,n_2\sim N/d\\(n_1, n_2)=1}}\hspace*{-0.4cm}\frac{1_{S_q(Y)}(dn_1)\beta(dn_2)}{n_2}\hspace*{-0.3cm}\sum_{\st{(u,v,w)\in\mc{A}\\w\neq 0, c\neq 0}}\widehat{\phi}\lr{\frac{3Mu}{q}}\widehat{g}\lr{w/n_2;n_1, n_2, c}e\lr{\frac{c\overline{n_1} w}{n_2}}.
\end{align}

We first show that the contributions from the sum $S_7(\boldsymbol{\beta})$ are negligible.

\begin{lem}\label{lem_for_S7}
    Let $S_7(\boldsymbol{\beta})$ be as in \eqref{def_S7_new}. Suppose that the condition \eqref{con_range1} holds. Then, we have 
    \begin{align*}
        S_7(\boldsymbol{\beta})\ll R^{-2026}.
    \end{align*}
\end{lem}

\begin{proof}
    First note that if $(u, v, w)\in \mathbb{Z}^3\setminus\mathcal{A}$, then $u$, $v$, $w$ must satisfy one of the following:
    \begin{align*}
        |u|\geq qR^{\eta}/M, \quad |v|\geq N/R^{1-4\eta}, \quad \text{or} \quad |w|\geq 8NR/q.
    \end{align*}
    If $|w|\geq 8NR/q$, then if $n_2\sim N/d$, we have
    \begin{align*}
        \bigg|\dfrac{w}{n_2}\bigg|\geq \dfrac{4Rd}{q}\geq \dfrac{4R}{q}.
    \end{align*}
    So, we can apply \cite[Lemma 5.7]{Irv2014} to deduce that $\widehat{g}(w/n_2; n_1, n_2, c)=0$. Therefore, $S_7(\boldsymbol{\beta})=0$ in this case.

Without loss of generality, we can now assume that $|w|<8NR/q$ and consider the following set $\mathcal{A}_1$ given by
\begin{align*}
    \mathcal{A}_1=\{(u, v)\in \mathbb{Z}^2\colon |u|\geq qR^\eta/M\: \text{or}\: |v|\geq N/R^{1-4\eta}\}.
\end{align*}
It is enough to show that for any positive integer $A$,
\begin{align*}
    \sum_{(u, v)\in \mathcal{A}_1}\bigg|\widehat{\phi}\bigg(\dfrac{3Mu}{q}\bigg)\widehat{g}(w/n_2; n_1, n_2, c)\bigg|\ll_A \dfrac{1}{R^{A}}.
\end{align*}
By definition of $g$ from \eqref{def_g} and that of $c$ from \eqref{def_c}, we see that the left-hand side of the above expression is
\begin{align*}
    \ll  \int_{-\infty}^\infty\sum_{\st{(u,v)\in\mc{A}_1}}\bigg|\widehat{\phi}\lr{\frac{3Mu}{q}} \widehat{\phi}\lr{\frac{tR}{q}}\widehat{\phi}\lr{\frac{(vq-a''u-tn_1)R}{n_2q}}\bigg|dt.
\end{align*}
The key idea now is to use the fact that $|\widehat{\phi}(x)|\ll_A\min\{1, |x|^{-A}\}$ for any positive integer $A$. 

Observe that since $(u, v)\in \mathcal{A}_1$, let us first assume that $|u|\geq qR^\eta/M$. Then, this implies that $3M|u|/q\geq 3R^\eta$. So, the function $\widehat{\phi}(3Mu/q)$ makes the above sum bounded by $\ll R^{-A}$ for any large $A>0$.

Similarly, the presence of the function $\widehat{\phi}(tR/q)$ makes the above sum small for $|t|\geq qR^{-1+\eta}$. So, we may assume that
\begin{align*}
    |u|\leq \dfrac{qR^\eta}{M}, \quad |t|\leq \dfrac{q}{R^{1-\eta}}, \quad \text{and} \quad |v|\geq \dfrac{N}{R^{1-4\eta}}.
\end{align*}
In this case, we have
\begin{align*}
     \frac{(a''u+tn_1)R}{n_2q}\ll \bigg(\dfrac{qD_0}{N}\cdot \dfrac{qR^\eta}{M} + \dfrac{q}{R^{1-\eta}}\bigg)\dfrac{R}{q}\ll R^{3\eta},
\end{align*}
using the facts that $a''\leq q$, $d\leq D_0=R^{2\eta}$, $MN\asymp X=qR$. Therefore, we can deduce that
\begin{align*}
    \dfrac{(vq-a''u-tn_1)R}{n_2q}\gg R^{\eta}.
\end{align*}
Therefore, we can conclude that in this case, as the presence of the function $\widehat{\phi}\big((vqR-a''uR-tn_1R)/(n_2q)\big)$ makes our sum arbitrarily small. So, we can conclude that for any $A>0$,
\begin{align*}
     \sum_{(u, v)\in \mathcal{A}_1}\bigg|\widehat{\phi}\bigg(\dfrac{3Mu}{q}\bigg)\widehat{g}(w/n_2; n_1, n_2, c)\bigg|\ll_A \dfrac{1}{R^{A}}.
\end{align*}
This completes the proof.
\end{proof}

Now it remains to estimate the sum $S_8(\boldsymbol{\beta})$ given by \eqref{def_S8_new}. Interchanging the order of summations and bounding $|\beta(n)|\leq 1$ trivially, we see that
\begin{align}\label{S8_ineq}
    S_8(\boldsymbol{\beta})\ll \dfrac{MR^2}{q^2N}\sum_{1\leq d<D_0}d\sum_{\substack{(u, v, w)\in \mathcal{A}\\ w, \: c\neq 0}}\sum_{n_2\sim N/d}\bigg|\sum_{\substack{n_1\sim N/d\\ (n_1, n_2)=1}}1_{S_q(Y)}(n_1)\widehat{g}(w/n_2; n_1, n_2, c)e\bigg(\dfrac{c\overline{n_1}w}{n_2}\bigg)\bigg|.
\end{align}
Next, we wish to remove the factor $\widehat{g}$ from the above sum by using partial summation. Indeed, we write $f(x)=\widehat{g}(w/n_2; x, n_2, c)$. Then, by \eqref{def_g}, we have
\begin{align*}
f(x)=\int_{-\infty}^\infty \widehat{\phi}\bigg(\dfrac{tR}{q}\bigg)\widehat{\phi}\bigg(\dfrac{(c-tx)R}{n_2q}\bigg)e\bigg(\dfrac{-tw}{n_2}\bigg)\: dt.
\end{align*}
By \cite[Lemma 5.7]{Irv2014}, we have $f(x)\ll q/R$ for all $x\in [N/d, 2N/d]$. Also, note that since $\phi$ is a smooth function, we can conclude that $f$ is also a smooth function.  Then, differentiating with respect to $x$, we obtain
\begin{align*}
    \dfrac{df(x)}{dx}\ll \dfrac{R}{n_2q}\int_{-\infty}^{\infty}\bigg|\widehat{\phi}\bigg(\dfrac{tR}{q}\bigg)\widehat{\phi}^\prime\bigg(\dfrac{(c-tx)R}{n_2q}\bigg)|t|\: dt.
\end{align*}
Since $\widehat{\phi}(x)\ll_A \min\{1, |x|^{-A}\}$ for any positive integer, we can infer that the right-hand side of the above expression is
\begin{align*}
    \ll_\eta \dfrac{R}{n_2q}\int_{-qR^{-1+\eta/2}}^{qR^{-1+\eta/2}}|t|dt + R^{-2026}\ll_{\eta} \dfrac{qdR^{\eta}}{NR},
\end{align*}
using the fact that $n_2\sim N/d$. We may now apply the above bound together with the partial summation to deduce from \eqref{S8_ineq} that
\begin{equation}
\begin{aligned}
    S_8(\boldsymbol{\beta}) &\ll_{\eta} \dfrac{MR^{1+\eta}}{qN}\sum_{1\leq d<D_0}d\sum_{\substack{(u, v, w)\in \mathcal{A}\\w, \: c\neq 0}}\sum_{n_2\sim N/d}\max_{N'\sim N/d}\bigg|\sum_{\substack{N/d\leq n_1<N'\\ (n_1, n_2)=1}}1_{S_q(Y)}(n_1)e\bigg(\dfrac{c\overline{n_1}w}{n_2}\bigg)\bigg|\\
    &\ll_{\eta}  \dfrac{MR^{1+\eta}}{qN}(\log N) \sum_{1\leq d<D_0}d\sum_{\substack{(u, v, w)\in \mathcal{A}\\w, \: c\neq 0}}S_9,
\end{aligned}
\label{eq_s8_s9}
\end{equation}
where
\begin{align}\label{def_S9_new}
    S_9:=\max_{N'\sim N/d}\max_{N/d\leq N_1<N'/2}\sum_{n_2\sim N/d}\bigg|\sum_{\st{n_1\sim N_1\\(n_1,n_2)=1}}1_{S_q(Y)}(n_1)e\lr{\frac{c\overline{n_1} w}{n_2}}\bigg|.
\end{align}
Note that trivially, we have $S_9\ll Nd^{-1}\Psi(N/d, Y)$. Therefore, our task reduces to obtaining sharp estimates for the sum $S_9$ above to complete the estimates for $S_8(\boldsymbol{\beta})$. Before that, we make the following remark on the size of $cw$, which will be crucial in our analysis later on.

\begin{re}\label{re_sizeof_cw}
Recall from \eqref{def_c}, we have $c=vq-a''u$. Then, if $(u, v, w)\in \mathcal{A}$, we deduce that  \begin{align*}
    |cw|\leq \dfrac{N}{R^{1-4\eta}}\cdot q \cdot \dfrac{8NR}{q} + q\cdot \dfrac{qR^\eta}{M}\cdot \dfrac{8NR}{q}\ll N^2R^{4\eta},
\end{align*}
using the facts that $1\leq a''\leq q$ and $MN\asymp X=qR$.
\end{re}

\subsection{Evaluation of the sum $S_8(\boldsymbol{\beta})$} 

Our goal is to bound the sum $S_8(\boldsymbol{\beta})$ for $Y=(\log X)^{C(\delta)}$. In particular, we will apply Theorem \ref{thm_kloos} to estimate the sum $S_9$.

\begin{lem}\label{lem_S9_new}
    Let $S_9$ be as in \eqref{def_S9_new}. Let $0<\tau<1$ be such that $Y\leq N^{1-\tau}<N/d$. Then, for any $\eta>0$, we have
    \begin{align*}
        S_9\ll_\eta X^{7\eta}\bigg(\dfrac{N^{2-\tau/2}Y^{1/2}}{d^{3/2}} + \dfrac{N^{7/4+\tau/4}}{d^2} + \dfrac{N^{2-\tau}}{d}\bigg).
    \end{align*}
\end{lem}

\begin{proof} 
 We apply Theorem \ref{thm_kloos} with $a=cw$ together with Remark \ref{re_sizeof_cw} to obtain
    \begin{align*}
        S_9\ll_\eta X^{7\eta}\bigg\{\bigg(\dfrac{N}{d}\bigg)^{3/2}Y^{1/2}z^{1/2} + \bigg(\dfrac{N}{d}\bigg)^2\dfrac{1}{z^{1/4}} + \dfrac{Nz}{d}\bigg\},
    \end{align*}
   where $z$ is any real number satisfying $Y\leq z<N/d$. Choosing $z=N^{1-\tau}$ for some $0<\tau<1$ such that $Y\leq z<N/d$, we deduce that
    \begin{align*}
        S_9\ll_\eta X^{7\eta}\bigg(\dfrac{N^{2-\tau/2}Y^{1/2}}{d^{3/2}} + \dfrac{N^{7/4+\tau/4}}{d^2} + \dfrac{N^{2-\tau}}{d}\bigg),
    \end{align*}
    as desired.
    \end{proof}

    We are now ready to estimate the sum $S_8(\boldsymbol{\beta})$.

    \begin{lem}\label{lem_for_S8}
    Let $S_8(\boldsymbol{\beta})$ be as in \eqref{def_S8_new}. Let $Y=( \log X)^{C(\delta)}$ for some large real number $C(\delta)>0$. Let $M, N\geq 2$ be such that the condition \eqref{con_range1} holds. Then, for any $0<\eta<\delta/20$, we have 
    \begin{align*}
        S_8(\boldsymbol{\beta})\ll_{\delta}  \frac{R^{4-2\eta}}{M}.
    \end{align*}
\end{lem}

\begin{proof}
    By \eqref{eq_s8_s9}, we have
    \begin{align*}
        S_8(\boldsymbol{\beta})\ll_\eta \dfrac{MR^{1+\eta}\log N}{qN}\sum_{1\leq d<D_0}d \cdot \#\mathcal{A}\max_{\substack{(u, v, w)\in \mathcal{A}\\ w, \: c\neq 0}}S_9.
    \end{align*}
    where $\mathcal{A}$ is given by \eqref{def_A} and $S_9$ is given by \eqref{def_S9_new}. Note that we have
    \begin{align*}
        \#\mathcal{A}\ll \dfrac{qR^\eta}{M}\cdot \dfrac{N}{R^{1-4\eta}}\cdot \dfrac{NR}{q}\ll \dfrac{N^2R^{5\eta}}{M}.
    \end{align*}
 Next, observe that since $Y=(\log X)^{C(\delta)}$ and $q/R^{1-\delta}\leq N\leq R^{12/11-\delta}$ by \eqref{con_range1}, it is evident that for any $0<\tau<1$, we have $Y\leq N^{1-\tau}\leq N/d$. So, invoking Lemma \ref{lem_S9_new}, we see that
    \begin{align*}
        S_8(\boldsymbol{\beta})\ll_{\eta}\dfrac{MR^{1+2\eta}}{qN}\cdot \dfrac{N^2R^{5\eta}}{M}\sum_{1\leq d\leq D_0}d\bigg(\dfrac{N^{2-\tau/2}Y^{1/2}}{d^{3/2}} + \dfrac{N^{7/4+\tau/4}}{d^2} + \dfrac{N^{2-\tau}}{d}\bigg)X^{7\eta}.
    \end{align*}
    Recalling that $D_0=R^{2\eta}$, we have
    \begin{align*}
        S_8(\boldsymbol{\beta})\ll_\eta\dfrac{NR^{1+16\eta}}{q}(N^{2-\tau/2}Y^{1/2}+N^{7/4+\tau/4}).    
\end{align*}
Note that $(7+\tau)/4\leq 2-\tau/2$ if and only if $\tau\leq 1/3$. Choosing $\tau=1/3$, we infer that
\begin{align*}
    S_8(\boldsymbol{\beta})\ll_\eta \dfrac{MNR^{1+16\eta}}{qM}N^{11/6}Y^{1/2}\ll_\eta  \dfrac{R^{2+16\eta}}{M}N^{11/6}Y^{1/2},
\end{align*}
using the fact that $MN\asymp X=qR$. By \eqref{con_range1}, we have $N\ll R^{12/11-\delta}$. This implies that
\begin{align*}
    S_8(\boldsymbol{\beta})\ll_{\eta}\dfrac{R^{4-11\delta/6 + 16\eta}}{M}Y^{1/2}\ll_\delta \dfrac{R^{4-2\eta}}{M},
\end{align*}
using the fact that $Y=(\log X)^{C(\delta)}\ll X^{o(1)}$ and $0<\eta<\delta/20$. This completes the proof.
\end{proof}

\subsection{End step: Proof of Proposition \ref{Prop: combine S1 and S2}} We are now ready to complete the proof of Proposition \ref{Prop: combine S1 and S2}.

\begin{proof}[Proof of Proposition \ref{Prop: combine S1 and S2}]
Recall that if $M, N\geq 2$ satisfy the condition \eqref{con_range1}, we wish to show that
\begin{align*}
        S(\boldsymbol{\beta})=\dfrac{3\widehat{\phi}(0)^3K(N, Y)MNR^2}{q^2}\sum_{n\sim N}\beta(n) + O_\delta\bigg(\dfrac{R^{4-2\eta}}{M}\bigg),
    \end{align*}
    where $S(\boldsymbol{\beta})$ is given by \eqref{def_S_beta}.
    
    First, by \eqref{S_1'=S1+S2}, Lemma~\ref{lem_S2_new}, and \eqref{def2_S1}, we have 
\begin{align}\notag 
S(\boldsymbol{\beta})=S_3(\boldsymbol{\beta})+S_4(\boldsymbol{\beta})+O_{\delta}\lr{\frac{R^{4-2\eta}}{M}}.
\end{align}
Next, by Lemma~\ref{lem_S3_new}, we have
$$S_3(\boldsymbol{\beta})= \frac{3\widehat{\phi}(0)^3K(N,Y)MNR^2}{q^2}\sum_{n\sim N}\beta(n)+O_{\delta}\lr{\frac{R^{4-2\eta}}{M}}.$$
So, to complete the proof, it is enough to show that
\begin{align}\label{S1'_in_proof}
    S_4(\boldsymbol{\beta})\ll_\delta\dfrac{R^{4-2\eta}}{M}.
\end{align}
Indeed, by \eqref{S4=S5+S6}, we have $S_4(\boldsymbol{\beta})=S_5(\boldsymbol{\beta}) + S_6(\boldsymbol{\beta})$. Next, by  Lemma~\ref{lem_for_S5}, we have
\begin{align*}
    S_5(\boldsymbol{\beta})\ll_\delta \dfrac{R^{4-2\eta}}{M}.
\end{align*}
Finally, by \eqref{def2_S6}, Lemma~\ref{lem_for_S7}, and Lemma \ref{lem_for_S8}, we have
\begin{align*}
    S_6(\boldsymbol{\beta})=S_7(\boldsymbol{\beta}) + S_8(\boldsymbol{\beta})\ll_\delta\dfrac{R^{4-2\eta}}{M}.
\end{align*}
Combining the above estimates, the relation \eqref{S1'_in_proof} follows. This completes the proof.
\end{proof}

\section{Proof of Theorem~\ref{main_thm2}}\label{sec: Proof of main theorem 2}
Recall that we wish to estimate the sum
\begin{align*}
    \Sigma(q, R):=\sum_{\substack{X/4\leq n\leq 4X}}1_{S_q(Y)}(n)\sum_{\substack{r\in \mathbb{Z}\\ na\equiv r\Mod q}}\phi\bigg(\dfrac{r}{R}\bigg),
\end{align*}
where $X=q^{2/(1+\theta)}$, $R=q^{(1-\theta)/(1+\theta)}$ with $\theta$ given by \eqref{def-theta} and $Y=(\log X)^{C(\varepsilon)}$ for some large $C(\varepsilon)>0$. 

We will apply Propositions \ref{Prop: Main 1} and \ref{Prop: Main 2} to estimate the above sum $\Sigma(q, R)$. To do so, we need to verify if the hypotheses in these propositions are satisfied by the above values of $X, R$, and $\theta$. First note that if $\theta=6/17-\varepsilon$ for any small $\varepsilon>0$, we have
\begin{align*}
    \dfrac{1+\theta}{1-\theta}=\dfrac{23}{11}-2\bigg(\dfrac{17}{11}\bigg)^2\varepsilon + O(\varepsilon^2).
\end{align*}
This implies that $q/R^{1-\delta}\leq R^{12/11-\delta}$ holds for any $0<\delta<(17/11)^2\varepsilon$.

Let $\delta$ be a parameter such that $0<\delta<2\varepsilon$. Then, the hypotheses in Propositions  \ref{Prop: Main 1} and \ref{Prop: Main 2} hold, and in particular, the ranges of $M$ and $N$ are non-empty. Suppose that $M, N\geq 2$ are two real numbers that satisfy the condition \eqref{con_range1}.
Then, by positivity, we have
\begin{align*}
    \Sigma(q, R)\gg \sum_{m\sim M}1_{S_q(Y)}(m)\sum_{n\sim N}1_{S_q(Y)}(n)\Phi(mn, R).
\end{align*}
Applying Propositions \ref{Prop: Main 1} and \ref{Prop: Main 2}, for any small $0<\eta<\varepsilon/100$, we deduce that
\begin{align}\label{eq_proof_thm_main}
    \Sigma(q, R)\gg\dfrac{\widehat{\phi}(0)R}{q}\bigg(\sum_{m\sim M}1_{S_q(Y)}(m)\bigg)\bigg(\sum_{n\sim N}1_{S_q(Y)}(n)\bigg) + O_\varepsilon(R^{2-\eta}).
\end{align}
Note that by \eqref{con_range1}, we have the following crude bounds
\begin{align*}
    M\gg \dfrac{X}{N}\gg \dfrac{qR}{R^{2-\delta}}\gg R^\delta =X^{\delta(1-\theta)/2} \quad \text{and} \quad N\geq \dfrac{q}{R^{1-\delta}}\geq R^\delta = X^{\delta(1-\theta)/2}.
\end{align*}
In particular, we may apply Lemma \ref{lem_for_interval_small_y}. Indeed, we choose $C(\varepsilon)> 50(1-\theta)\varepsilon^{-1}+1$ so that $Y=(\log X)^{C(\varepsilon)}$. Then, by Lemma \ref{lem_for_interval_small_y}, we have
\begin{align}\notag 
    \bigg(\sum_{m\sim M}1_{S_q(Y)}(m)\bigg)& \bigg(\sum_{n\sim N}1_{S_q(Y)}(n)\bigg) \\
  \notag  &\gg \Psi(M, Y)\Psi(N, Y)\prod_{\substack{p|q\\p\leq Y}}\bigg(1-\dfrac{1}{p^{\alpha(M, Y)}}\bigg)\bigg(1-\dfrac{1}{p^{\alpha(N, Y)}}\bigg)\\
  \label{eq_proof_thm_main_i}  &\gg X\bigg(\dfrac{\log X}{\log Y}\bigg)^{-(1+o(1))\log X/\log Y}\prod_{\substack{p|q\\p\leq Y}}\bigg(1-\dfrac{1}{p^{\alpha(M, Y)}}\bigg)\bigg(1-\dfrac{1}{p^{\alpha(N, Y)}}\bigg),
\end{align}
by \eqref{Eq: smooth crude} and the fact that $MN\asymp X$. In order to simplify the above expression, note that since $Y=(\log X)^{C(\varepsilon)}$, we have
\begin{align}\notag 
    \bigg(\dfrac{\log X}{\log Y}\bigg)^{-(1+o(1))\log X/\log Y}
    & \gg \exp\bigg(-\big(1+o(1)\big)\dfrac{\log X}{C(\varepsilon)\log \log X}\log \bigg(\dfrac{\log X}{C(\varepsilon)\log \log X}\bigg)\bigg)\\
    \label{eq_proof_thm_main_ii}& \gg X^{-1/C(\varepsilon) +o(1)}.
\end{align}
Secondly, we observe that since $\alpha(M, Y), \alpha(N, Y)\in [1/2, 1)$ by \eqref{estimates_for_alpha_x_y}, we have for some explicit constant $c>0$,
\begin{align}\notag 
    \prod_{\substack{p|q\\p\leq Y}}\bigg(1-\dfrac{1}{p^{\alpha(M, Y)}}\bigg)\bigg(1-\dfrac{1}{p^{\alpha(N, Y)}}\bigg)
    &\gg \prod_{p|q}2^{-1}\cdot 2^{-1}= 4^{-\omega(q)}\\
   \label{eq_proof_thm_main_iii} &\gg 4^{-c(1+\theta)\log X/(2\log \log X)} =X^{o(1)}=R^{o(1)},
\end{align}
since $\omega(q)\ll \log q/\log \log q$, $X=q^{2/(1+\theta)}$, and $R=X^{(1-\theta)/2}$. 

Hence, combining \eqref{eq_proof_thm_main}, \eqref{eq_proof_thm_main_i}, \eqref{eq_proof_thm_main_ii}, and \eqref{eq_proof_thm_main_iii} together with the fact  $X=qR$, we may conclude that
\begin{align*}
    \Sigma(q, R)\gg_\varepsilon R^{2}(qR)^{-1/C(\varepsilon) + o(1)} + O_\varepsilon(R^{2-\eta}) \gg R^{2-\frac{1-\theta}{2C(\varepsilon)}+o(1)},
\end{align*}
by using the fact that $q=R^{(1+\theta)/(1-\theta)}$. This completes the proof.

    \bibliographystyle{plain}

\end{document}